\documentclass[12pt,leqno,twoside]{amsart}
\usepackage{amssymb,verbatim,url,graphicx}
\usepackage{accents}
\usepackage[percent]{overpic}
\usepackage{accents}
\usepackage{graphicx,color}
\usepackage{psfrag}
\usepackage[english]{babel}  
\usepackage{babelbib}
\renewcommand{\baselinestretch}{1.15} 

\graphicspath{{figs/}}

\newtheorem{theorem}{Theorem}[section]

\newtheorem{lemma}[theorem]{Lemma}
\newtheorem{proposition}[theorem]{Proposition}
\newtheorem{corollary}[theorem]{Corollary}

\theoremstyle{definition}
  \newtheorem*{definition}{Definition}
  
\theoremstyle{remark}
  \newtheorem*{remark}{Remark}
  \newtheorem*{remarks}{Remarks}
  \newtheorem*{examples}{Examples}

\newcommand{\ia}{{\rm({\it i\/}\rm)}}
\newcommand{\ii}{{\rm({\it ii\/}\rm)}}
\newcommand{\iii}{{\rm({\it iii\/}\rm)}}

\newcommand{\Isom}{\mbox{Isom}}
\newcommand{\ahalf}{{\textstyle\frac12}}
\renewcommand{\epsilon}{\varepsilon}
\newcommand{\id}{\operatorname{id}}

\newcommand{\modtwo}{\, \mbox{\small{$\bmod$}}\,2}
\newcommand{\modfour}{\, \mbox{\small{$\bmod$}}\,4}
\newcommand{\modtwotin}{\:\!\mbox{\tiny{$\bmod \:\!2$}}\,}
\newcommand{\N}{\mathbb{N}}
\let\paragraph=\S  
\newcommand{\R}{\mathbb{R}}
\renewcommand{\S}{\mathbb{S}}

\newcommand{\Z}{\mathbb{Z}}

\begin{document}

\title[Embedded periodic surfaces]
{Construction of embedded periodic surfaces in $\R^n$}
\date{\today}
\author[Grosse-Brauckmann, K\"ursten]{Karsten Grosse-Brauckmann, 
Susanne K\"ursten}
\address{Technische Universit\"at Darmstadt, Fachbereich Mathematik (AG 3),
    Schlossgartenstr.~7, 64289 Darmstadt, Germany}
\email{kgb@mathematik.tu-darmstadt.de}

\begin{abstract}
  We construct embedded minimal surfaces which 
  are $n$-periodic in $\R^n$.  They are new for codimension $n-2\ge 2$.
  We start with a Jordan curve of edges of the $n$-dimensional cube.  
  It bounds a Plateau minimal disk which Schwarz reflection extends 
  to a complete minimal surface.
  Studying the group of Schwarz reflections, we can characterize 
  those Jordan curves for which the complete surface is embedded.
  For example, for $n=4$ exactly five such Jordan curves generate 
  embedded surfaces.
  Our results apply to surface classes other than minimal as well, 
  for instance polygonal surfaces.
\end{abstract}
\maketitle

\section{Introduction} 

Triply periodic embedded minimal surfaces in euclidean $3$-space
are a common model for real-world interfaces.
Riemann and Schwarz were the first to construct such surfaces.
They chose a suitable polygonal Jordan curve, 
constructed a minimal disk it bounds,
and used the Schwarz reflection principle
to extend the surface by successive half-turn rotations.
Many further polygonal contours have been considered since, 
and the use of the Plateau solution instead of 
the Weierstrass representation simplified the construction,
see \cite[\paragraph~818]{nit}.
A.~Schoen and Karcher used the method together with conjugation to
establish the existence of many further triply 
periodic embedded minimal surfaces \cite{schoen}\cite{kar}.  
The original method of Riemann and Schwarz was also applied 
in the $3$-sphere to construct compact embedded minimal surfaces,
namely by Lawson in 1970~\cite{law},
and recently by Choe and Soret \cite{chso}.
\smallskip

In the present paper we
construct $n$-periodic embedded surfaces in $\R^n$, $n\ge 3$,  
with codimension $n-2$ by the same method.
To generate them, we consider a special class of Jordan curves,
as well as surfaces they bound:
\begin{definition}
  \ia\ Let $C=[-\frac12,\frac12]^n\subset\R^n$ be the unit cube.
  A \emph{Jordan path} is an embedded edge loop $\Gamma\subset\partial C$ 
  which contains at least one edge in each coordinate direction.\\
  \ii\ An \emph{initial surface} $\Sigma_0\subset C$ 
  is an embedded, compact surface with interior $\accentset{\circ}{\Sigma}_0 \subset (-\frac12,\frac12)^n$ whose boundary
  $\Gamma:=\partial \Sigma_0$ is a Jordan path. 
\end{definition}
As we show in Sect.~\ref{se:min}, 
for all Jordan paths~$\Gamma$ of interest to us, 
the minimal disk obtained as a Plateau solution for~$\Gamma$
is an example of an initial surface~$\Sigma_0$.
Another example would be a triangulated, piecewise linear disk,
obtained as the cone of~$\Gamma$ over the origin. 
Similarly, there is a vertex in the cube such that 
the cone of~$\Gamma$ is a triangulated discrete minimal surface.
Schwarz reflection across the straight edges of~$\Gamma$
maintains all these surface classes so that 
by successive application we obtain a complete surface~$\Sigma$
which is again minimal, triangulated, or discrete minimal.
We consider the following problem: \emph{Is $\Sigma$ embedded?}
\smallskip

We start by studying the symmetry group of~$\Sigma$
in Sect.~\ref{secgroupS}.  
The resulting lattice of translations $\Lambda(\Sigma)$ contains $(4\Z)^n$
by Thm.~\ref{4znlattice}, 
and so it is convenient to study the symmetries in the quotient
$T^n:=\R^n/(4\Z)^n$, see Sect.~\ref{se:quotients}.
As our main result, in Sect.~\ref{se:charemb} we arrive at
two theorems characterizing the embeddedness of~$\Sigma$.
First, Thm.~\ref{thmfilledcubes} characterizes it in terms 
of the number of copies of the initial surface~$\Sigma_0$ in $T^n$, 
namely exactly $2^{n+2}$ out of the $4^n$~unit cubes 
making up for~$T^n$ must contain a copy of~$\Sigma_0$.
Second, Thm.~\ref{criterion} characterizes embeddedness
in terms of the number of lattice elements for the quotient surface in~$T^n$.
These results represent the nontrivial answer to the embeddedness problem
for codimension $n-2\ge 2$, and also shed new light on 
the known case of codimension~$1$.

Since $\Sigma$ is generated by Schwarz reflection with respect to~$\Gamma$,
the number of copies of~$\Sigma_0$ in~$T^n$ depends on~$\Gamma$ alone.
This raises the specific question of which Jordan curves 
meet the conditions of our main theorems. 
We present an algorithmic answer in Sect.~\ref{seccriterion},
where we depend on the notation for Jordan curves 
introduced in Sect.~\ref{se:J}. 
We exemplify our answer on the particular case $n=4$ 
in Sects.~\ref{ss:r4m8} and~\ref{se:rfourcase}:
From the six Jordan paths with $8$~edges, exactly three
lead to an embedded surface~$\Sigma$ (Thm.~\ref{eightedgesembedded}), 
and there are exactly 
two more such Jordan paths with ten or more edges (Thm.~\ref{theoremR4}).  
That is, altogether five Jordan paths 
lead to embedded periodic surfaces in~$\R^4$. 
For all higher dimensions we exhibit families of Jordan curves 
which lead to embedded surfaces $\Sigma\subset\R^n$, 
see Prop.~\ref{pr:series}.
Finally, in Sect.~\ref{se:min}, we confirm that minimal surfaces
satisfy our assumptions so that all admissible Jordan curves 
indeed generate embedded periodic minimal surfaces.

Let us remark that our results also apply to the case that the
cube is replaced by a rectangular box.  On the other hand,
our assumptions are restrictive in that we only admit 
Jordan curves which are edge loops of the cube or box.
It is crucial for the present work that 
Schwarz reflections preserve coordinate directions.
More general cases, such as diagonal edges, remain to be studied.

Let us also note that for the case of minimal surfaces in~$\R^n$, 
only very few explicit $n$-periodic examples with $n\ge 4$
seem to be known:
Shoda uses Weierstrass data to construct
examples of genus 3 and~10 in~$\R^4$ without discussing embeddedness
\cite{sh1}, \cite{sh2}.  Nagano and Smyth give some general theory
and an existence statement for $n$-periodic immersions~\cite{nasm}.  
\smallskip

This work contains the results of a PhD thesis of the second author,
written at TU Darmstadt~\cite{kue}, 
for which the first author was the advisor.
We would like to thank Steffen Fr\"ohlich (Mainz) for suggesting 
the problem and helpful discussions.
We are also grateful to Michael Joswig (TU Berlin)
for discussions about the polygonal case.

\section{Jordan paths} \label{se:J}


We can describe a path $\Gamma$ of length $m\ge 2n$ along the edges 
of $C\subset\R^n$
in terms of an initial vertex~$p$ of~$C$ and a sequence 
$\gamma=(\gamma(i))_{1 \leq i \leq m}$ with $1 \leq \gamma(i) \leq n$:
The path starts at~$p$ along the edge in the $\gamma(1)$-direction, 
then follows an edge in $\gamma(2)$-direction, etc., 
see Fig.~\ref{figpathR3}.
\begin{figure}[t]
\begin{center}
\hspace{3mm}
\raisebox{8mm}{
\setlength{\unitlength}{1cm}
\begin{picture}(1.8,1.5)
\put(0,0){\vector(1,0){1}}
\put(1.2,-0.1){$x_1$}
\put(0,0){\vector(0,1){1}}
\put(-0.1,1.3){$x_3$}
\put(0,0){\vector(1,1){0.8}}
\put(1.0,0.8){$x_2$}
\end{picture}
}
\includegraphics[scale=0.47]{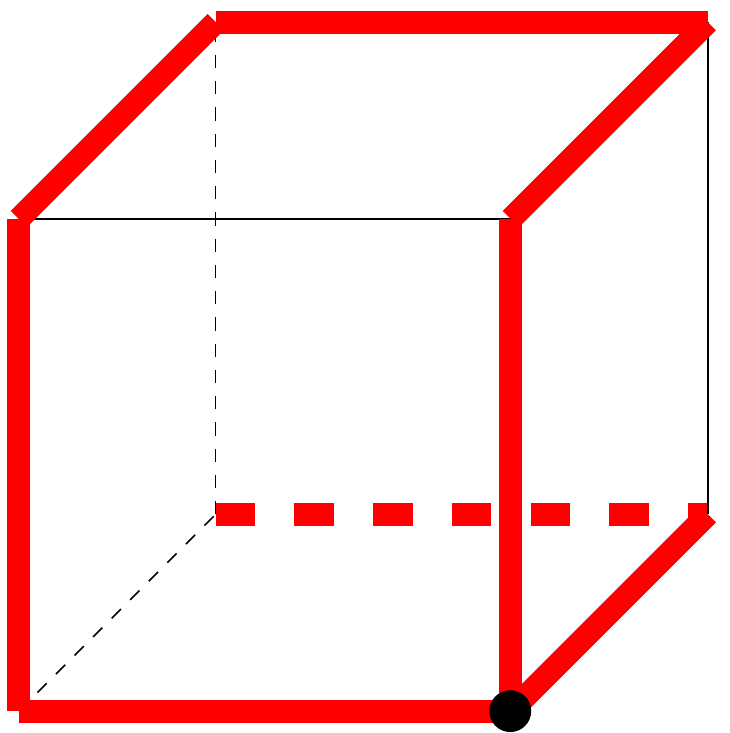}
\hspace{7mm}
\raisebox{8mm}{
\setlength{\unitlength}{1cm}
\begin{picture}(1.8,1.5)
\put(0,0){\vector(1,0){1}}
\put(1.2,-0.1){$x_1$}
\put(0,0){\vector(0,1){1}}
\put(-0.1,1.3){$x_3$}
\put(0,0){\vector(2,3){0.7}}
\put(0.5,1.2){$x_4$}
\put(0,0){\vector(1,1){0.8}}
\put(1.0,0.8){$x_2$}
\end{picture}
}
\includegraphics[scale=0.5]{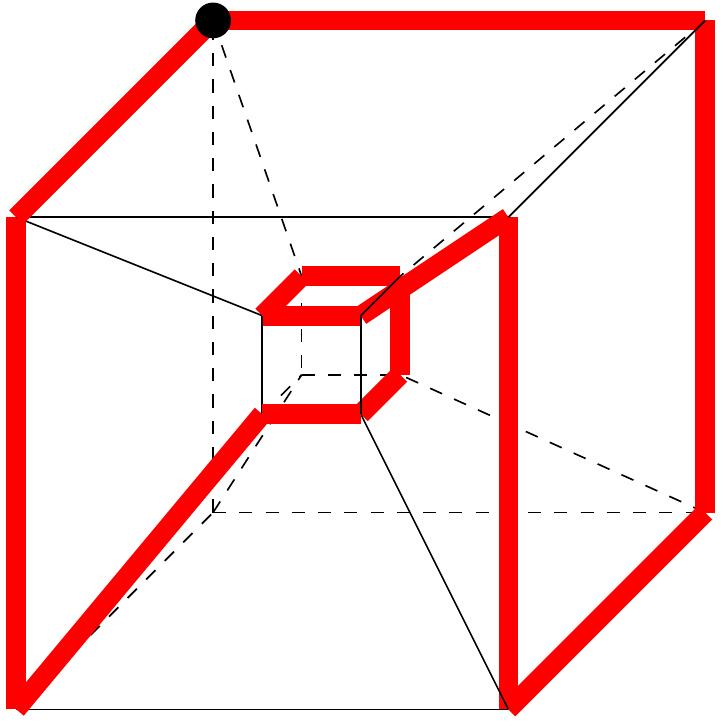} 
\end{center}
\caption{\label{figpathR3} 
Left: A non-closed, non-embedded path in $\R^3$ with marked initial vertex 
$p= (\frac{1}{2}, -\frac{1}{2}, -\frac{1}{2})$ and sequence $3212\,3121$.
Right: The projection of the Jordan path in $\R^4$ with initial vertex 
$p= (-\frac{1}{2}, \frac{1}{2}, \frac{1}{2}, -\frac{1}{2})$ 
and sequence $132\, 3412\, 1321\, 432$.}
\end{figure}
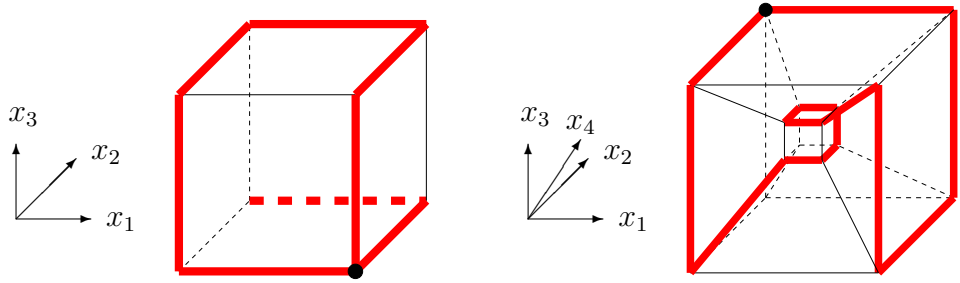
We need to impose the following
conditions in order for $\bigl(p,(\gamma(i))_{1 \leq i \leq m}\bigr)$ 
to describe an embedded edge loop:  
\begin{itemize}
 \item Closedness: The path contains an even number of edges
   in each coordinate direction, that is, $\gamma(i) =\beta$ holds 
   for an even number of indices $i$ for each $1 \leq \beta \leq n$. 
\item Embeddedness: The path does not contain a proper closed 
   subpath, that is, 
   no proper consecutive subsequence of~$\gamma$ (considered as cyclic)  
   satisfies the first condition. 
\end{itemize}
\begin{figure}[b]
\begin{center}
  \includegraphics[scale=0.45]{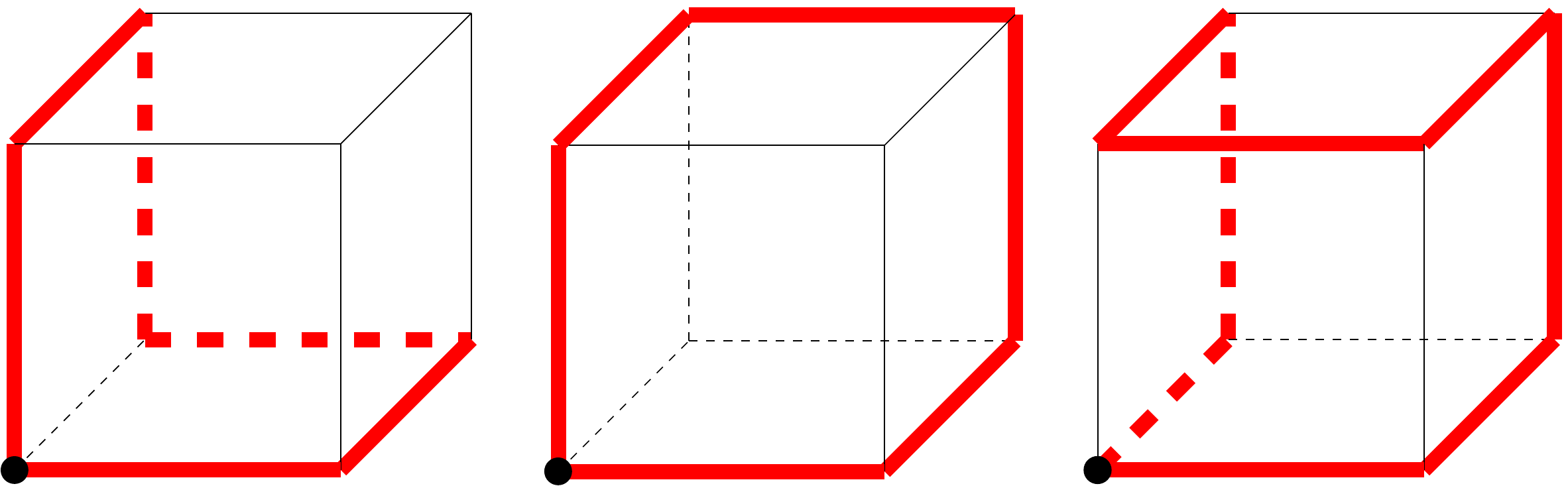} 
\end{center}
  \caption{\label{fi:JordpathR3} Jordan paths in $\R^3$:
  $\Gamma_{CLP}=121\,323$, $\Gamma_{D}=123\,123$, 
  and, related to the Gergonne problem, $\Gamma_{GP}=1232\,1232$.}
\end{figure}

For dimension $3$, this gives the following paths:
\begin{proposition} \label{dim3}
  Up to symmetry, there are exactly the three Jordan 
  paths in $\R^3$ depicted in Fig.~\ref{fi:JordpathR3}. 
\end{proposition}
\noindent
For each of these Jordan paths, the unique Plateau disk 
determines an initial surface~$\Sigma_0$, 
whose extension yields a triply periodic minimal surface,
which is known to be embedded (see also Prop.~\ref{dimension3}).
All three surfaces are due to Schwarz.
In Schoen's nomenclature~\cite{schoen}
these are the $CLP$-surface and the $D$-surface. 
The extension of the third surface turns out to be identical 
with a deformation of the $D$-surface, namely the cube as shown 
in the second image of Fig.~\ref{fi:JordpathR3} is deformed to 
a rectangular box with height $\sqrt 2$ over a unit square 
(compare~\cite[5.1.3]{kar}).
It also arises from the extension of Schwarz' solution 
to Gergonne's problem.
%

To give a systematic treatment of Jordan paths in higher dimensions, 
we need to distinguish them only up to symmetry.
It can be checked that the following operations 
on a Jordan path $\Gamma=\bigl(p,(\gamma(i))_{1 \leq i \leq m}\bigr)$ are
equivalent to the action of symmetries:
\begin{itemize}
 \item Change of initial vertex $p$, 
 \item cyclic shift of the sequence $\gamma$, 
 \item order reversal of the sequence $\gamma$, or
 \item the action of a permutation on the values of $\gamma$.
\end{itemize}
%
Therefore, we identify Jordan paths $\Gamma$ 
with a sequence $(\gamma(i))_{1 \leq i \leq m}$ ($1 \leq \gamma(i) \leq n$),
which we regard as cyclic, non-oriented, and modulo permutation of the values.

The above operations can be used to verify that two Jordan paths are equal.
On the other hand, let us describe an invariant useful 
to distinguish Jordan paths up to symmetry.
For each value of $\beta$, consider the numbers $i_1<\ldots< i_k$ 
such that $\gamma(i_j)=\beta$.  
These define a cyclic vector
of gap lengths $a(\beta):=(|i_{j+1}-i_j| \mid j=1,\ldots,k)$.
The $n$ vectors $a(1),\ldots,a(n)$, up to permutation, 
form an invariant of the Jordan curve. 
This follows from the fact that a symmetry maps 
parallel edges onto parallel edges.

Following this strategy, we can classify Jordan paths up to symmetry, 
by inspection of the finite number of sequences~$\gamma$, for instance:  
%
\begin{figure}[b]
\begin{center}
\includegraphics[scale=0.5]{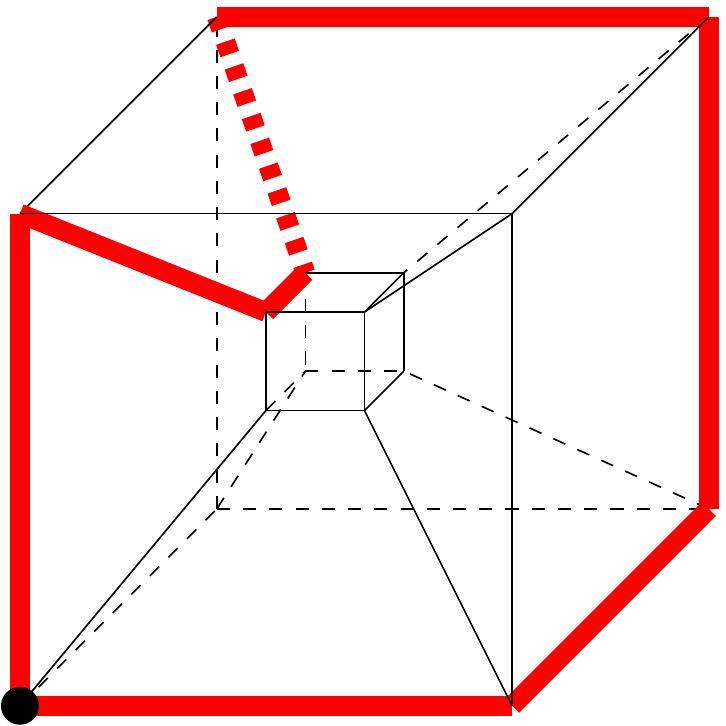}\hspace{3mm}
\includegraphics[scale=0.5]{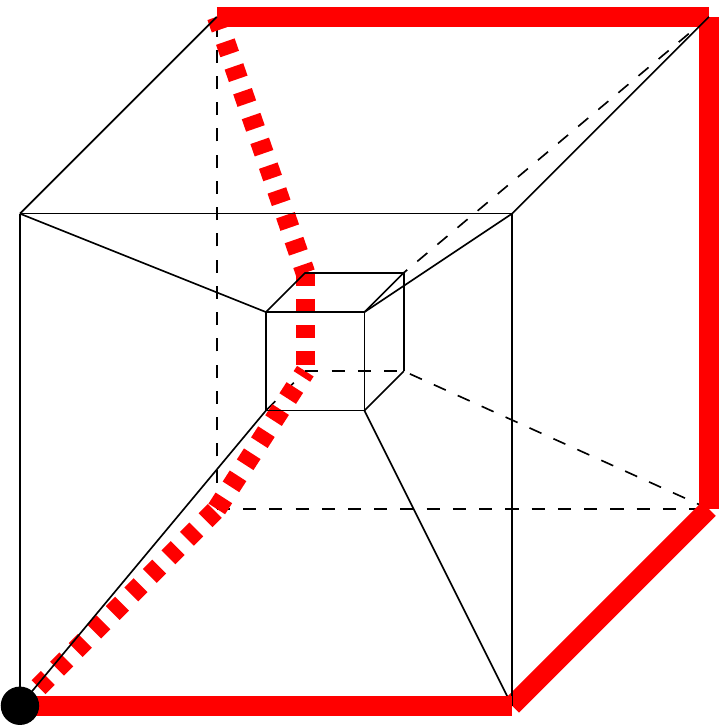}\hspace{3mm}
\includegraphics[scale=0.5]{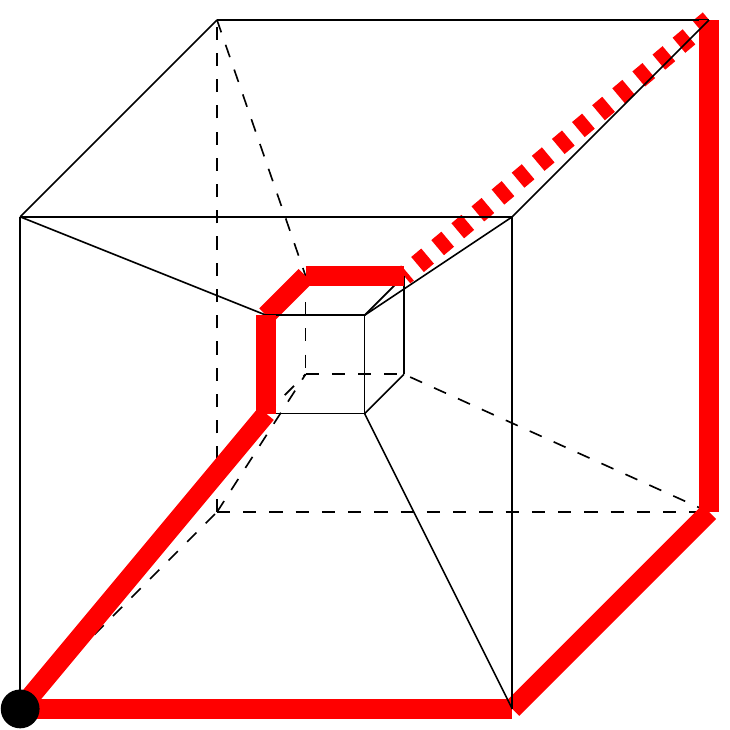} 
\end{center} 
\caption{\label{fi:R4nonembedded} 
  Three of the six Jordan paths with $8$ edges in $\R^4$,
  namely $\Gamma_1=1231\,4243$, $\Gamma_2=1231\,4342$,
  and $\Gamma_5=1234\,1234$.}
\end{figure}
\begin{proposition} \label{eightedges}
  In $\R^4$ there are exactly the six different Jordan paths $\Gamma$ 
  with length $m=8$,
  depicted in Fig.~\ref{fi:R4nonembedded} and~\ref{fi:R4embedded}, namely
\begin{eqnarray*}
&& 
\Gamma_1 := 1231\,4243, \hspace{1 cm}      
\Gamma_2 := 1231\,4342, \hspace{1 cm}      
\Gamma_3 := 1231\,4234, \\ &&              
\Gamma_4 := 1231\,4324, \hspace{1 cm}      
\Gamma_5 := 1234\,1234, \hspace{1cm}       
\Gamma_6 := 1232\,1434 \, .                
\end{eqnarray*} 
\end{proposition}

\section{The symmetry group and location of self-intersections of $\Sigma$}
\label{secgroupS}

\subsection{Schwarz reflection and the group $S$}
Let $1 \leq \beta \leq n$.
Schwarz reflection w.r.t.\ an edge in the $\beta$-direction 
is a half-turn rotation, preserving the $\beta$-coordinate and
acting as a reflection with respect to all other coordinate directions. 
To describe a Schwarz reflection, 
denote a half-turn rotation about the $x_\beta$-axis by
$$ 
  \rho^\beta\colon \R^n \to \R^n, \qquad 
  x \mapsto (-x_1,  \ldots, -x_{\beta-1}, x_\beta, -x_{\beta+1}, \ldots, -x_n) .
$$
Moreover, denote translations in the form $\tau_v (x) := x+v$. 
Then a Schwarz reflection fixing the line 
$\{q+ t e_\beta\mid t \in \R\}$ through $q \in \R^n$ is given by
\begin{equation} \label{formschwrefl} 
  \tau_{2q - 2 q_\beta  e_\beta} \circ \rho^\beta \, .\end{equation}

From an initial surface $\Sigma_0$ we obtain a complete surface~$\Sigma$
by successive Schwarz reflections across boundary edges.
The surface $\Sigma$ consists of infinitely many isometric copies 
of $\Sigma_0$, each contained in some cube $C^v:=\tau_v(C)$
of the cube tesselation $\{C^v\mid v \in \Z^n\}$ of~$\R^n$. 
We call these isometric copies of~$\Sigma_0$ 
the \emph{surface patches} of~$\Sigma$.

In order to analyze the complete surface $\Sigma$ 
we use a group of symmetries of~$\Sigma$:
\begin{definition} 
  Let $S$ be the subgroup of $\Isom(\R^n)$ generated by 
  the $m$ Schwarz reflections across the edges of
  the Jordan path $\Gamma:= \partial \Sigma_0$. 
\end{definition}
\noindent
Clearly $S$ is a subgroup of the symmetry group of~$\Sigma$.  
It is a proper subgroup if and only if 
the initial surface $\Sigma_0$ is invariant 
under a nontrivial symmetry $s\not\in S$ of $C$.
In the present section we will analyze $S$ and relate elements of $S$ 
to surface patches of $\Sigma$.

To discuss generators of $S$, consider an edge of $\Gamma$
in the $\beta$-direction.  Its midpoint $q$ has coordinates 
\begin{equation*} 
  q_\beta=0\qquad\text{ and }\qquad 
  q_\alpha\in \bigl\{\pm\textstyle\frac12\bigr\}\text{ for } \alpha\not=\beta.
\end{equation*}
Setting $u:=2q\in\Z^n$ in \eqref{formschwrefl} we find that the edge
gives rise to a generator of form
\begin{equation} \label{formgenS} 
  \{\tau_{u} \circ \rho^\beta\mid 
     u_\beta=0,\, u_\alpha\in\{\pm 1\} \text{ for all } \alpha\not=\beta\}.
\end{equation} 

For any symmetry $\rho$ of the cube $C$ it is easy to check 
$$ 
  \rho \circ \tau_v = \tau_{\rho(v)} \circ \rho.
$$ 
Consequently, for two symmetries $\rho$ and $\sigma$ of $C$, and $u,v \in \Z^n$,
\begin{equation} \label{formcomprefl} 
  \tau_{u} \circ \rho \circ \tau_{v} \circ \sigma
  = \tau_{u+ \rho (v)} \circ (\rho\circ \sigma).
\end{equation}
This, together with \eqref{formgenS}, 
means that each element of $S$ has the form
\begin{equation} \label{structureS}
  \tau_v \circ \rho
\end{equation}
 with $v\in \Z^n$ and $\rho \in \langle \rho^1, \ldots, \rho^n \rangle$. 

\subsection{$S$ as a semidirect product}
In order to identify $S$ with a subgroup of a semidirect product
let us first identify the rotation $\rho^\beta$ with the element 
\begin{equation}\label{rhodef}
  \rho^\beta \in \Z_2^n,\quad\text{ with components }\quad 
  \rho^\beta_\beta = 0 \quad\text{and}\quad  
  \rho^\beta_\alpha =1 \text{ for }\alpha \neq \beta. 
\end{equation}
Here $\rho\in\Z_2^n$ acts on $\R^n$ by the coordinatewise sign change
$x\mapsto(-1)^\rho x: 
= \bigl( (-1)^{\rho_\alpha} x_\alpha \bigr)_{1\leq \alpha\leq n}$.
Note that composition of two rotations $\rho^{\beta_1} ,\rho^{\beta_2}$ 
agrees with addition in $\Z_2^n$. 

Furthermore we define the group
$$ H:= \langle \rho^1, \ldots, \rho^n\rangle\subset\Z_2^n,$$ 
where the composition is addition in $\Z_2^n$. 
For later use we claim 
\begin{equation} \label{formH} 
  H=\left \{ \begin{array}{ccl} \Z_2^n 
       & \mbox{ for } & n \mbox{ even,}\vspace{0.2 cm}\\
  \{ \rho \in \Z_2^n \, | \, \sum_{\alpha=1}^n \rho_\alpha =_{\Z_2} 0 \} 
       & \mbox{ for } & n \mbox{ odd.}  \end{array} \right .
\end{equation}
Indeed, for $n$ even, 
the relation ''$\subseteq$'' is clear from the definition of $H$.
To verify the same relation for $n$ odd, recall from \eqref{rhodef}
that each generator $\rho^\beta$ has vanishing component sum, 
which is preserved under composition.
The other relation ''$\supseteq$'' can be verified by writing 
generators of the right-hand side as a sum of suitable $\rho^\beta$'s. 
%

Now we define 
$$ G:= \Z^n \rtimes H, 
$$ 
where the group operation is
\begin{equation}\label{compositionG} 
  (u, \rho) \circ_G (v, \sigma) 
  = \bigl(u +_{\Z_n} (-1)^{\rho} v,\; \rho +_{\Z_2^n} \sigma \bigr).
\end{equation} 

Due to \eqref{formcomprefl}, \eqref{structureS}, and \eqref{compositionG},
we can identify $S$ with a subgroup of $G$ by the injective homomorphism 
$S \to G$, $\tau_v \circ \rho \mapsto (v,\rho) $. 
Let us refer to $T(v,\rho):=v$ as the \emph{translational part} 
of $(v,\rho)\in S$ and to $R(v,\rho):=\rho$ as its \emph{rotational part}.
In terms of the identification we can state:
\begin{lemma} \label{lepU} 
  $S$ is a subgroup of 
  $$ U:= \{(v,v\modtwo)\mid v \in \Z^n, v \modtwo \in H \}\subset G \: .$$
\end{lemma}
\begin{proof}
As \eqref{compositionG} shows, $U$ is closed under the composition of $G$. 
Furthermore the generators $(u,\rho^\beta)$ of $S$ 
satisfy $u \modtwo = \rho^\beta$,
see \eqref{formgenS} and \eqref{rhodef}, so they are elements of $U$. 
Altogether we get $S \subset U$.
%
\end{proof}

Now we are ready to prove $S$ can be identified with the set of surface patches.
We define a map $\Phi$ from $S$ to the set of surface patches by
associating to each symmetry $s\in S$ the surface patch $s(\Sigma_0)$.
\pagebreak[3]
\begin{proposition} \label{propbijpatches}
  $\Phi$ is a bijection between the group $S$ and the set of
  surface patches of~$\Sigma$.
\end{proposition}
\begin{proof}
  By Lemma~\ref{lepU} two distinct elements $s_1 \neq s_2$ of $S$ are 
  of the form $s_1=(u,u\modtwo)$ and $s_2 = (v,v\modtwo)$ with $u \neq v$. 
  As $u$ and $v$ represent translational parts, 
  $s_1$ maps $\Sigma_0$ to a surface patch in a cube $C^u$, 
  and $s_2$ maps $\Sigma_0$ to a surface patch in $C^v \neq C^u$.  
  Hence $\Phi(s_1) \neq \Phi(s_2)$ and $\Phi$ is injective.

Let $\tilde{\Sigma}_0$ be a surface patch.  
By definition of $\Sigma$  
there exist Schwarz reflections $s_{g_1}, \ldots, s_{g_k}$ 
with $(s_{g_k} \circ \ldots \circ s_{g_1}) (\Sigma_0) = \tilde{\Sigma}_0$,
subject to the following:
The line $g_j$ is fixed under $s_{g_j}$,
points in a coordinate direction, and contains a boundary edge 
of $(s_{g_{j-1}} \circ \ldots \circ s_{g_1})(\Sigma_0)$. 
It can be checked that any two Schwarz reflections about lines $g,h$
pointing in coordinate directions satisfy the commutation relation 
$s_g \circ s_{{h}} = s_h \circ s_{s_h({g})}$. 
Iterated application of the relation proves the existence of generators  
$s_1, \ldots, s_k$ of $S$ such that 
$s_{g_{k}} \circ \ldots \circ s_{g_1} = s_1 \circ \ldots \circ s_k \in S$. 
So indeed $\tilde{\Sigma}_0$ is the image of $\Sigma_0$ under a symmetry in $S$ and $\Phi$ is surjective. 
\end{proof}
\noindent  From now on we will identify the
symmetry group $S$ with the subgroup $S \subset U$ and with 
the surface patches of~$\Sigma$.

\subsection{Location of self-intersections}

By Prop.~\ref{propbijpatches}, 
elements $s = (v,\rho) \in S$ are in 1-1 relation to surface patches
$\Phi(s)=(\tau_v \circ \rho)(\Sigma_0)$ contained in the cube $C^v$.
On the other hand, Lemma~\ref{lepU} implies that the
translational part $v$ of $s$ determines its rotational part $\rho=v\modtwo$.
We conclude:  
\begin{theorem} \label{th:atmostone}
  Each cube $C^v$ contains at most one surface patch of $\Sigma$.
  Specifically, $C^v$ contains a patch if and only if $(v,v\modtwo) \in S$.
\end{theorem}
\noindent
We call a cube $C^v$ which contains a surface patch a \emph{filled cube};
we use this terminology also for elements $(v,v\modtwo) \in S$.

By Thm.~\ref{th:atmostone}, self-intersections of $\Sigma$ 
cannot occur in the interior of any cube $C^v$.
Therefore, self-intersections of $\Sigma$ can only occur on cube boundaries,
and so depend on the Jordan path $\Gamma= \partial \Sigma_0$ alone. 
As $\Gamma$ consists of a union of edges, if self-intersections occur at all
then they occur on entire edges, in particular on their bounding vertices.
Moreover, at each vertex, the angle of the contour is $90^\circ$,
and upon Schwarz reflection the
incident surface patches arise in multiples of four. 
We obtain the following characterization of embeddedness: 
\begin{corollary} \label{co:intersectionvertex}
 \ia\ The surface $\Sigma$ is embedded if and only if it has no 
 self-intersections at the vertices 
 of the cube tessellation $\{C^v\mid v\in \Z^n\}$ of $\R^n$.\\
  \ii\
  A vertex $p$ of the cube tessellation is a self-intersection point 
  of~$\Sigma$ if and only if there are at least eight surface patches 
  of~$\Sigma$ which contain $p$ as a boundary point.
\end{corollary}

For odd dimension $n$ we can  combine Lemma~\ref{lepU} with \eqref{formH} 
to see that half of the cubes must be empty (not filled):
Equivalently, this can be seen geometrically, 
by considering a checkerboard black-and-white
colouring of the cube tessellation.  
In odd dimensions, 
Schwarz reflection respects the colouring and so one colour remains empty.  
This implies in particular that our problem for $n=3$ has an affirmative 
answer, no matter which Jordan path is considered:
\begin{proposition} \label{dimension3}
  In $\R^3$ each extended surface $\Sigma$ is embedded.
\end{proposition}
\begin{proof}
  Suppose that $\Sigma$ is not embedded.
  By Cor.~\ref{co:intersectionvertex} 
  there is a vertex $p$ such that at least eight distinct 
  surface patches of $\Sigma$ contain~$p$.
  On the other hand, in dimension $n=3$ there are eight cubes incident to $p$,
  and only half of them can be filled, a contradiction.
\end{proof}

For higher dimension, however, the argument given in the proof fails, 
and in even dimension all cubes may be filled anyway. 
For instance, in $\R^4$, possibly each of the $16$ cubes 
incident to a vertex is filled.

\section{The lattice of $\Sigma$ and the quotient of $S$} 
\label{se:quotients}

Our goal is to read off the condition of Cor.~\ref{co:intersectionvertex} 
from the given Jordan path $\Gamma$.
For that end we need to study the lattice generated by $\Gamma$ in detail.

\subsection{The lattice of $\Sigma$}
We define the \emph{lattice $\Lambda(\Sigma)$} of the extended surface $\Sigma$ by
$$\Lambda(\Sigma) :=  \{ v \in \Z^n \, | \, \tau_v(\Sigma)=\Sigma \} \, .$$ 
Clearly, $\Lambda(\Sigma)$ is an additive group. 
We observe that for $v\in\Lambda(\Sigma)$
the translations $\tau_v$ are elements from $S$, 
and that they may change the orientation of~$\Sigma$
(see Sect.~\ref{sectionorientable}).

\begin{lemma} \label{latticeel}
  If $(u,\rho), (v,\rho) \in S$ then $u-v \in \Lambda(\Sigma)$.
\end{lemma}
\begin{proof}
As the rotational parts of $(u,\rho)\in S$ and $(v,\rho)\in S$ agree, 
the corresponding surface patches in $ C^{u}$ and $ C^{v}$ 
agree up to the translation $\tau_{u-v}$, that is, 
$u - v \in \Lambda(\Sigma)$.
\end{proof}

Our Jordan paths $\Gamma$ contain a pair of distinct edges for each 
coordinate direction.  
The corresponding Schwarz reflections are of the form
$(u,\rho)$, $(v,\rho)$ in $S$ with $u-v\in (2\Z)^n$, see~\eqref{formgenS}. 
Thus the lemma exhibits at least $n$ nontrivial elements
of $\Lambda(\Sigma)$ which are contained in $(2\Z)^n$. 
While in general $\Lambda(\Sigma)$ does not contain the entire set $(2\Z)^n$, 
the following can be asserted independently of $\Gamma$:
\begin{theorem} \label{4znlattice}
  For each initial surface $\Sigma_0\subset\R^n$ 
  we have $(4\Z)^n \subset \Lambda(\Sigma)$.
\end{theorem}
\begin{proof}
  Our strategy is to use the lemma to prove 
  $\{4e_\beta\mid \beta=1,\ldots,n\}\subset \Lambda(\Sigma)$, 
  which implies the claim.
  By a coordinate renumbering it suffices to verify this for $\beta=1$.

  By definition, $\Gamma$ contains an edge $e$ in the $1$-direction. 
  It has the form
  $$ 
    e= \bigl\{(t,q_2,  \ldots, q_n)  \mid 
       t \in \big[\!\!-\ahalf, \ahalf \big ] \bigr\},
    \text{ where }q_\alpha \in \bigl\{\pm \ahalf \bigr\} 
    \text{ for }2 \leq \alpha \leq n \: .
  $$ 
Let $f$ and $g$ be the two edges of $\Gamma$ incident to $e$, 
where $f$ starts at $\big(\frac{1}{2},q_2, \ldots, q_n\big)$. 
Again by coordinate renumbering, 
we may assume $f$ points in the $2$-direction so that
$$
  f = \bigl \{(\ahalf,t, q_3, \ldots, q_n) \mid  
      t\in\bigl[-\ahalf,\ahalf \bigr] \bigr \}.
$$ 
Let $(u, u\modtwo)\in S$ be the generator of $S$ corresponding to $f$;
according to \eqref{formschwrefl} then $u = (1,0,2 q_3, \ldots, 2 q_n)$.
We distinguish two cases:

\noindent Case 1: If $g$ is parallel to $f$ then
$$
  g = \bigl\{(-\ahalf,t, q_3, \ldots, q_n)  \mid
       t \in \bigl[-\ahalf, \ahalf \bigr] \bigr \},
$$
and we let $(v, v\modtwo)\in S$ 
be the corresponding generator of $S$. By  \eqref{formschwrefl} we find $v =(-1,0,2 q_3, \ldots, 2 q_n)$.
Since $u\modtwo = v\modtwo$ in this case,  
Lemma~\ref{latticeel} implies $ u-v = 2 e_1 \in \Lambda(\Sigma)$, 
and so in particular $4 e_1 \in \Lambda(\Sigma)$.

\noindent Case 2: Suppose $g$ is not parallel to $f$. 
We may assume $g$ points in the $3$-direction.  Then 
$$  
  g = \bigl\{(-\ahalf,q_2, t, q_4, \ldots, q_n)  \mid 
       t \in \bigl[-\ahalf,\ahalf \bigr] \bigr\} 
$$
and again we let $(v, v\modtwo)\in S$ be the corresponding generator of $S$, 
where $ v = (-1,2 q_2, 0, 2 q_4, \ldots, 2 q_n)$. 
As $2 q_i \in \{\pm 1\}$, we find
\begin{equation*}
\begin{split}
 (u, u\modtwo) \circ (v, v \modtwo) &= \bigl(u+(-1)^{u \modtwotin} v, \,
    (u + v)\modtwo\bigr) \\
&= \bigl((2, 2q_2, 2 q_3, 0, \ldots, 0),(u + v)\modtwo\bigr) ,\\
 (v, v\modtwo) \circ (u, u\modtwo) &= \bigl(v+(-1)^{v \modtwotin} u, \,
    (u+v) \modtwo\bigr) \\
&= \bigl((-2, 2q_2, 2 q_3, 0, \ldots, 0),(u + v)\modtwo\bigr) \, .
\end{split}
\end{equation*} 
Lemma~\ref{latticeel} implies $4e_1 \in \Lambda(\Sigma)$.
\end{proof}

\subsection{Quotient groups}  

By Thm.~\ref{4znlattice} we may pass to the quotient under the group $(4\Z)^n$. 
Then the ambient $n$-torus $T^n:=\R^n/(4\Z)^n$
contains the quotient surface $\Sigma^Q:=\Sigma/(4\Z)^n$,
with quotient lattice 
\[
  \Lambda^Q(\Sigma):= \Lambda(\Sigma)/(4\Z)^n
  \;\subset\; \Z^n/(4\Z)^n = \Z_4^n.
\]

In particular, we want to consider quotients of $U$ and its subgroup~$S$.
The subgroup $\{ (w, \id) \mid w \in (4\Z)^n\}$ 
of both $S$ and~$U$ is easily seen to be a normal subgroup,
and so defines quotient groups $U^Q$ and $S^Q$.
We can identify the cosets of $U$ with
\begin{align*} U^Q &:= \{(v\modfour, v\modtwo) \mid v \in \Z^n,\, v\modtwo \in H \} \\
&= \{(v, v\modtwo) \mid v \in \Z_4^n,\, v\modtwo \in H \} \, . 
\end{align*}
The group operation of $U^Q$ is given by \eqref{compositionG},
except that the translational part is taken modulo~$4$;
similarly for the subgroup $S^Q\subset U^Q$.
Recall that generators of $S$ correspond to edges of $\Gamma$. 
Taking their translational part modulo $4$ we obtain generators of~$S^Q$.
By \eqref{formgenS} these generators are of form $(v, \rho^\beta)$ where 
$v_\beta= 0$, while all other coordinates of~$v$ equal $1$ or~$3$.

By passing to the quotient we obtain groups with useful algebraic properties.
For instance, while $U$ is not abelian 
(see the calculation in the proof of Thm.~\ref{4znlattice}, Case~2) 
its quotient $U^Q$ is:
\begin{lemma} \label{commutative}
  The finite group $U^Q$ is abelian and each element of~$U^Q$ 
  is self-inverse; the same holds for its subgroup~$S^Q$.
\end{lemma}
\begin{proof}
  Consider two arbitrary elements $(u, u\modtwo)$ 
  and $(v, v\modtwo)$ of $U^Q$ and their compositions 
\begin{equation}\label{uvversusvu}
\begin{split}
 (u, u\modtwo) \circ (v, v \modtwo) &= \bigl(u+(-1)^{u \modtwotin} v, \,
    (u + v)\modtwo\bigr)  ,\\
 (v, v\modtwo) \circ (u, u\modtwo) &= \bigl(v+(-1)^{v \modtwotin} u, \,
    (u+v) \modtwo\bigr) \, .
\end{split}
\end{equation}
  Thus $U^Q$ is abelian if 
\begin{equation}\label{eqabelian}  
  u- (-1)^{v\modtwotin} u - \bigl(v - (-1)^{u\modtwotin}  v\bigr) =  0 
\end{equation}  
holds in $\Z_4^n$ for all $u,v \in \Z_4^n$.
Now the $\alpha$-coordinate of $u- (-1)^{v\modtwotin} u$ is, in $\Z_4$,
$$
  \left ( u- (-1)^{v\modtwotin}  u  \right )_\alpha  = \begin{cases}
  0  & \mbox{for } u_\alpha\modtwo = 0 \mbox{ or } v_\alpha\modtwo = 0\, , \\
  2  &
  \mbox{for } u_\alpha\modtwo = v_\alpha\modtwo = 1 .
  \end{cases} 
$$ 
For each $\alpha$, the right hand side is symmetric in $u$ and $v$,
and so indeed \eqref{eqabelian} holds.

To prove each element $(v, v\modtwo) \in U^Q$ is self-inverse, 
note $U^Q$ is generated by elements corresponding to Schwarz reflections 
across edges of~$C$.  Clearly, Schwarz reflection is self-inverse.
But $U^Q$ is an abelian group, so all its elements are self-inverse.
Alternatively, using \eqref{uvversusvu} for $u=v$ gives the element
$\bigl(v+_{\Z_4^n} (-1)^{v\modtwotin}  v , 0\bigr)$
whose components can be seen to vanish.

Since $S^Q$ is a subgroup of $U^Q$, these properties hold for $S^Q$ as well.
\end{proof}

Let us finally consider the order
\begin{equation}
   |U^Q| = \bigl| \{(v, v\modtwo) \mid 
        v \in \Z_4^n,\, v\modtwo \in H \}\bigr| .
\end{equation}
By \eqref{formH}, for $n$ even $H=\Z_2^n$ and so 
$|U^Q| = |\Z_4^n| = 4^n = 2^{2n}$; for $n$ odd, however,
the constraint stated in \eqref{formH} gives 
$|U^Q| = \frac{1}{2}|\Z_4^n| = 2^{2n-1}$.  We have proved:
\begin{proposition} \label{pr:order}
  The order $|U^Q|$ is $2^{2n}$ for $n$ even and $2^{2n-1}$ for $n$ odd.
  Consequently, the number $|S^Q|$ of filled cubes in $T^n$ is a power of two, 
  where $|S^Q|\leq 2^{2n}$ in even dimensions 
  and $|S^Q|\leq 2^{2n-1}$ in odd dimensions.
\end{proposition}
\noindent
We can interpret the result for $S^Q$ as follows. 
In even dimension each of the $4^n$ cubes in $T^n$ is possibly filled, 
while for odd dimension only cubes with one colour can be filled.

\section{Cubes of edge length 2}

Consider a vertex $p$ of the cube tesselation 
which is contained in a surface~$\Sigma$.
It must be incident to four surface patches related by Schwarz reflection.
If $\Sigma$ is embedded then no other patch can be incident to~$p$. 
Let us refer to cubes of edge length~$2$ as \emph{large cubes}. 
Then we can say that for an embedded surface we expect 
the large cube with midpoint~$p$ contains only four filled cubes. 

Making the additional hypothesis that each of the $2^n$ large cubes
tesselating the quotient torus~$T^n$ contains four filled cubes, 
we arrive at the conjecture that embeddedness of~$\Sigma$ 
is equivalent to the number of filled cubes in~$T^n$ being exactly 
$|S^Q| = 4 \cdot 2^n = 2^{n+2}$. 
This conjecture will be proven only in the next section. 
Clearly, it is valid in dimension $n=3$.
Indeed, only half of the 64~cubes of~$T^3$ are filled, 
which gives $|S^Q|= |U^Q|= 32 =2^{n+2}$ filled cubes,
consistent with the fact that all our~$\Sigma\subset\R^3$ 
are embedded (Prop.~\ref{dimension3}). 
However, for dimensions $n\ge 3$, 
the number of cubes in~$T^n$ is $2^{2n}$,
and so the conjecture implies that for embedded surfaces~$\Sigma$ 
the density of filled cubes must decrease as $n$ grows.

The present section contains preparational lemmas.
We tesselate $T^n$ with large cubes, each containing
$2^n$ cubes of the original tesselation.
Our goal is to show that all large cubes contain 
the same number of filled cubes; the crucial step will be
to prove that this number does not vanish.
For our proof to work we need to consider a specific large cube tesselation 
which is adjusted to the position of the surface.

Let us set up some notation.
We represent $T^n$ with the cube 
$ C_4 := \bigl[-\frac{1}{2},\frac{7}{2}\bigr]^n 
$ 
of edge length $4$.  
We want to subdivide $C_4$ into $2^n$ large cubes.
Let $C_2 :=  \bigl[-\frac{1}{2},\frac{3}{2}\bigr]^n$
and associate to it the set of filled cubes 
$$
  L := \bigl\{(v, v\modtwo) \in S^Q \, | \, v \in \{0,1\}^n \bigr\}\,. 
$$
More generally, let $(2\Z_4)^n$ denote the subset of $\Z_4^n$ 
with even coordinates, endowed with $\Z_4^n$ addition,
and associate to $a \in (2\Z_4)^n$ 
the set of filled cubes in the large cube $\tau_a(C_2)$,
$$
  L^a := \bigl\{(v, v\modtwo) \in S^Q \, | \, (v-a) \in \{0,1\}^n \bigr\}\,;
$$
in particular $L^0=L$.  Then we can write
\begin{equation}\label{SQunionLb}
  S^Q = \mathbin{\dot{\bigcup}}_{a \in (2\Z_4)^n} L^a.
\end{equation}

To prove $|L^a| = |L|$ for all $a \in (2\Z_4)^n$ we will show
that $L$ is a subgroup of $S^Q$, while the $L^a$ are cosets of $L$. 
Thus all $L^a$ have the same number of elements. 
We will depend on the following technical fact.
\begin{lemma} \label{lemmacoset}
  If $a,b \in (2\Z_4)^n$ and
  $w \in L^{a}$, $z \in L^{b}$ then
  $ w \circ  {z} \in L^{a+b}$.
\end{lemma}
\begin{proof}
  Writing
  $w = (u, u\modtwo) \in S^Q$, ${z} = ({v}, {v}\modtwo) \in S^Q$  
  with  $u, {v} \in \Z_4^n$ we have
  \begin{equation} \label{eqdifference} 
    u-a, \; v-b \in \{0,1\}^n \: . 
  \end{equation} 
  According to \eqref{uvversusvu} 
  the composition $w \circ {z}$ has
  translational part $u+(-1)^{u\modtwotin}{v}$.
  We need to verify the equation in $\Z_4^n$
  \begin{equation}\label{firstcompwz}
    u+(-1)^{u\modtwotin}{v} - (a+b) \in \{0,1\}^n .
  \end{equation} 
  By \eqref{eqdifference}, the $\beta$-coordinate of
  $u-a$ is either $0$ or $1$.
  We distinguish these cases.

  \noindent \emph{Case} $(u-a)_\beta = 0$: 
  As $a \in (2\Z_4)^n$ we must have $u_\beta\modtwo = 0$ and so
  the $\beta$-coordinate of~\eqref{firstcompwz} reads
  $$ 
    u_\beta-a_\beta +{v}_\beta - {b}_\beta 
    = {v}_\beta - {b}_\beta.
  $$
  By \eqref{eqdifference} indeed this is in $\{0,1\}$.

  \noindent \emph{Case} $(u-a)_\beta = 1$:
  As $a \in (2\Z_4)^n$ we must now have $u_\beta\modtwo = 1$.  
  Hence the $\beta$-coordinate of \eqref{firstcompwz} reads
  $$
    u_\beta-a_\beta -{v}_\beta - {b}_\beta 
    = 1- ({v}_\beta-{b}_\beta) - 2 {b}_\beta .
  $$ 
  Using $2 {b}_\beta \bmod 4 = 0$ and \eqref{eqdifference} 
  we see that again this is in $\{0,1\}$.
\end{proof}

\begin{lemma} \label{numberLb}
  \ia\ $L$ is a subgroup of $S^Q$.\\
  \ii\ 
  Either $|L^a| = |L|$ or $|L^a| = 0$ for each $a \in (2\Z_4)^n$.
\end{lemma}
\begin{proof}
  \ia\ By Lemma~\ref{lemmacoset}, $L=L^0$ is closed, 
  and by Lemma~\ref{commutative}
  every element of $L \subset S^Q$ is self-inverse.
  
  \noindent
  \ii\
  Again by Lemma~\ref{lemmacoset} we can write
  $L^a=\bigcup_{u\in L^a}uL$.
  Any two cosets on the right are either equal or disjoint, 
  and they have the same number of elements. 
  Therefore, $L^a$ is a disjoint union of $k_a\in\N_0$ cosets 
  and $|L^a| = k_a |L|$.

  We claim that $k_a$ is either zero or one. 
  Indeed, if $L^a \neq \emptyset$ then there exists $u \in L^a$; 
  moreover, by Lemma~\ref{lemmacoset},
  composition with $u$ maps $L^a$ injectively to $L^{2a}=L$.
  So $|L^a| \leq |L|$ follows, implying our claim.
\end{proof}

We can now show that all large cubes contain the same
number $|L|$ of filled cubes.  Our proof needs the following
assumption, which can be achieved by applying an isometry of~$\R^n$:
\begin{equation}\label{pposition}
  p:=\bigl(\textstyle\frac{1}{2}, \ldots, \frac{1}{2}\bigr)
                   \in \Gamma\subset\Sigma.
\end{equation}
\begin{lemma} \label{Mequals}
  If\/ $\Sigma$ is positioned according to \eqref{pposition}
  then the following holds:\\
  \ia\ $L^a\not=\emptyset$ for all $a\in (2\Z_4)^n$,\\
  \ii\ $|L^a| = |L| $ for all $a\in (2\Z_4)^n$, and $|S^Q| = 2^n |L|$.
\end{lemma}
\begin{proof}
  First we prove \ia\ implies \ii: 
  The first statement of \ii\ follows from \ia\ and 
  Lemma~\ref{numberLb}.
  Therefore, we conclude that indeed
  $$  
    |S^Q| = \sum_{a \in (2\Z_4)^n} |L^a| 
    = \sum_{a\in (2\Z_4)^n} |L| = 2^n |L| \: . 
  $$
  
  It remains to prove \ia.  It will be useful to set
  $$M := \{a \in (2\Z_4)^n \mid L^a \neq \emptyset \} \,, 
  $$ 
  which is a subgroup of $(2\Z_4)^n$ by Lemma~\ref{lemmacoset}.
  Proving \ia\ is equivalent to showing $M = (2\Z_4)^n$.
  To see that, we will pick pairs of edges of~$\Gamma$
  and show that the composition of the corresponding two generators of $S^Q$ 
  lies in some $L^a$.  So it provides an element $a \in M$. 
  We will show that these elements generate~$(2\Z_4)^n$.

Consider two arbitrary generators $(u, u\modtwo)$ and 
$({v},{v}\modtwo)$ of $S^Q$, corresponding to two edges of $\Gamma$.
According to~\eqref{uvversusvu} their composition is
$$
  (w, w\modtwo) 
  := \bigl(u+(-1)^{u\modtwotin}{v} ,(u+v)\modtwo\bigr)
  \in S^Q,
$$
where we calculate mod~$4$.
By \eqref{SQunionLb} then $(w,w\modtwo)\in L^a$ for some
$a\in (2\Z_4)^n$, and so $a \in M$ follows.
By \eqref{formgenS}, the elements $u,v$ 
taken modulo~$4$ have coordinates in $\{0,1,3\}$.
Thus there are nine cases to consider for each 
coordinate $w_\beta$ of $w$.
In the following table each column represents such a case;  
in the last row we state the resulting value of $a_\beta$,
according to the definition of $L^a$.
\renewcommand{\arraystretch}{1.2}
\begin{equation*}\begin{array}{c|c|c|c|c|c|c|c|c|c}
u_\beta   & 0 & 0 & 0 & 1 & 1 & 1 & 3 & 3 & 3  \\
\hline 
{v}_\beta & 0 & 1 & 3 & 0 & 1 & 3 & 0 & 1 & 3 \\
\hline
w_\beta =u_\beta+(-1)^{u_\beta\modtwotin}{v}_\beta  &       
    0 & 1 & 3 & 1 & 0 & 2 & 3 & 2 & 0 \\
\hline
a_\beta = \begin{cases}
                0 & \mbox{for } w_\beta \in \{0,1\}\\
2 & \mbox{for } w_\beta \in \{2,3\}
               \end{cases}
&             0 & 0 & 2 & 0 & 0 & 2 & 2 & 2 & 0
  \end{array}
\end{equation*}
\renewcommand{\arraystretch}{1}
We conclude for the $\beta$-coordinate of $a$: 
\begin{equation}\label{aigleich}
  a_\beta=2 \quad\Leftrightarrow\quad 
  \text{ either } u_\beta=3\;\text{ or }{v}_\beta=3
\end{equation}
Recall from \eqref{formschwrefl} that $u_\beta=3$ 
if and only if the edge corresponding to~$(u, u\modtwo)$
has the $\beta$-coordinate constant with value $-1/2$;
similarly for~$v$.

Let us now establish the existence of pairs of edges as claimed.
By \eqref{pposition} we may assume 
$p:=(\frac{1}{2}, \ldots, \frac{1}{2})$ is the initial vertex of~$\Gamma$.
We pick an arbitrary orientation for $\Gamma$
and denote the resulting edge cycle by $(k_1, \ldots, k_m)$.
We suppose these edges point in the coordinate directions
$(\gamma(1), \ldots, \gamma(m))$ where $1\le \gamma(i)\le n$. 
Let $2\le i\le m$ and choose the pair of edges $k_1$ and $k_{i}$,
giving rise to some $a=a(i)\in M$ with coordinates as in the table.

Let us assume that among $k_1,\ldots,k_{m}$,
the edge $k_{i-1}$ is the first edge running
in the ${\gamma(i-1)}$-direction.
Then the $\gamma(i-1)$-coordinate attains the constant value~$-1/2$
on the subsequent edge~$k_i$, but not on~$k_1$. 
Hence application of~\eqref{aigleich} to $\beta=\gamma(i-1)$ gives
\begin{equation}\label{glb2} 
  a(i)_{\gamma(i-1)} = 2 \, .
\end{equation}

Without the assumption we can still constrain certain coordinates of~$a(i)$. 
On the one hand, along the edge $k_1$ no coordinate is fixed to~$-1/2$.
On the other hand, along the edge $k_i$ the 
$\beta$-coordinate can only be fixed to $-1/2$
if at least one of the edges $k_1,\ldots,k_{i-1}$ 
affects a change of the $\beta$-coordinate, that is, if it 
runs in the $\beta$-direction.
Let $e_\beta$ denote the unit vector in the $\beta$-direction,
then again by~\eqref{aigleich} we can conclude 
\begin{equation} \label{glbin} 
  a(i) \in \big\langle \{ 2 e_\beta \mid
 \beta \in \{\gamma(1), \ldots, \gamma(i-1) \} \}\big\rangle_{(2 \Z_4)^n} \, .
\end{equation} 

The preceding observations can be used to prove $2e_{\gamma(j)}\in M$ for 
$j=1,\ldots,m-1$ by an iterative argument.  Since $\gamma(j)$ attains
all values in $\{1,\ldots,n\}$, this will finish the proof.

To initialize, let $j=1$. We consider the edges $k_1,k_2$.
Then \eqref{glb2} and \eqref{glbin} immediately give
$2e_{\gamma(1)}=a(2)\in M$.
For the step, the iterative assumption 
is $2e_{\gamma(1)},\ldots, 2e_{\gamma(j-1)}\in M$ and 
we claim $2 e_{\gamma(j)} \in M$. 
In case one of the edges $k_1, \ldots, k_{j-1}$ runs in the 
$\gamma(j)$-direction the claim is contained in the iterative assumption.
Otherwise, \eqref{glb2} and~\eqref{glbin} give
\[
  a(j+1) \in \big\langle 2e_{\gamma(1)},\ldots,2e_{\gamma(j)}\big\rangle
  \quad\text{and}\quad a(j+1)_{\gamma(j)}=2.
\]
Therefore the vector $2e_{\gamma(j)}-a(j+1)$ is a linear combination of 
$2e_{\gamma(1)},\ldots,2e_{\gamma(j-1)}\in M$ and so lies in $M$ itself. 
Thus indeed $2e_{\gamma(j)}\in M$ holds, as desired.
\end{proof}

\section{Characterization of embeddedness}\label{se:charemb}

With the lemmas of the preceding section at hand, we can
now prove our main result.
\begin{theorem} \label{thmfilledcubes}  
  We have $|S^Q| \geq 2^{n+2}$,
  and the surface $\Sigma$ is embedded if and only if $|S^Q| = 2^{n+2}$,
  while $\Sigma$ is not embedded if and only if $|S^Q| \geq 2^{n+3}$.
\end{theorem}
\begin{proof}
  Without loss of generality we may assume \eqref{pposition}.
  Then on the one hand $|L|\geq 4$, 
  and on the other hand $|S^Q| = 2^n |L|$ by Lemma~\ref{Mequals}\ii.
  We have verified our first claim $|S^Q|\geq 2^{n+2}$.

  Therefore, to prove the embeddedness statements it suffices
  to show that $\Sigma$ is not embedded if and only if $|S^Q|\ge 2^{n+3}$.
Suppose first that $\Sigma$ is not embedded. 
As stated in Cor.~\ref{co:intersectionvertex}\ia\
a vertex of the cube tesselation
must be a self-intersection point of $\Sigma$; 
upon an isometry we may assume that this is the vertex~$p$.
By Cor.~\ref{co:intersectionvertex}\ii\ at least $8$~filled cubes 
must contain~$p$.  
Thus $|L| \geq 8$, and so $|S^Q| \geq 2^{n+3}$ by Lemma~\ref{Mequals}\ii. 

For the converse we assume $|S^Q| \geq 2^{n+3}$ 
and show that $\Sigma$ cannot be embedded.
By Lemma~\ref{Mequals}\ii\ we have $|L| \geq 8$. 
Each element of $L$ corresponds to a distinct symmetry $s=\tau_v \circ \rho$ 
with $v_\beta \in \{0,1\}$ and $\rho(x) = (-1)^{v\modtwotin} x$,
or equivalently to a surface patch $s(\Sigma_0)$.
We claim that each such $s$ fixes~$p$.  Indeed,
\begin{equation} 
  s(p)_\beta = ((\tau_v \circ \rho)(p))_\beta 
  = v_\beta + ((-1)^{v_\beta\modtwotin} p)_\beta  
  = \frac12,
\end{equation}
where the last equality reads either
$0 + \frac{1}{2} = \frac{1}{2}$ in case $v_\beta = 0$,
or $1-\frac{1}{2} = \frac{1}{2}$ in case $v_\beta = 1$.
But this implies $p$ is contained in eight distinct surface patches of $\Sigma$,
and so $\Sigma$ cannot be embedded by Cor.~\ref{co:intersectionvertex}.
\end{proof} 

The theorem gives a concise characterization of embeddedness. 
Nevertheless, in order to decide whether a given Jordan path
generates an embedded surface, a further characterization will
be useful, namely one in terms of the number of lattice elements.  
\begin{lemma} \label{thmlattice} 
  For $n$ even we have 
  $|S^Q| =  2^n |\Lambda^Q(\Sigma) \cap (2\Z_4)^n|$,
  while for $n$ odd we have
  $|S^Q|= 2^{n-1} |\Lambda^Q(\Sigma) \cap (2\Z_4)^n|$.
\end{lemma}
\begin{proof}
We order the elements of $S^Q$ by their rotational parts.
Let $\Omega(\rho)\subset S^Q$ be the subset of all elements 
with rotational part $\rho$. Then 
$$ 
  S^Q = {\dot{\bigcup_{\rho \in H}}}\; \Omega(\rho)\qquad 
  \mbox{and} \qquad |S^Q| = \sum_{\rho \in H} |\Omega(\rho)|.
$$ 
Recall that $\rho^\beta$ denotes a half-turn rotation with axis pointing
in the $\beta$-direction, and that $\Gamma$ contains edges in all directions.  
Thus $\Omega(\rho^\beta)$ is non-empty for each $1 \leq \beta \leq n$.
Upon composition of reflections we see that $\Omega(\rho)$ is non-empty
for each $\rho \in \langle \rho^1, \ldots, \rho^n \rangle = H$.

We claim each $\Omega(\rho)$ has the same number of elements. 
To see this, consider $(u,\rho) \in \Omega(\rho)\subset S^Q$.
We claim that 
\begin{equation*} 
  (v,\rho) \in \Omega(\rho)\quad\Longleftrightarrow\quad
  {u}-v \in \Lambda^Q(\Sigma)\cap (2\Z_4)^n \, . 
\end{equation*} 
Indeed, if $(v,\rho)\in\Omega(\rho)$ then
equality of rotational parts implies $u-v\in (2\Z_4)^n$ (see Lemma~\ref{lepU}), 
and Lemma~\ref{latticeel} implies $u-v\in\Lambda^Q(\Sigma)$.
Moreover, the converse is immediate.


Consequently, $|\Omega(\rho)| =  |\Lambda^Q(\Sigma) \cap (2\Z_4)^n|$ 
for every $\rho\in H$ and 
$$ |S^Q| = \sum_{\rho \in H} |\Omega(\rho)| 
   = |H| \cdot |\Lambda^Q(\Sigma) \cap (2\Z_4)^n|\, .
$$
Invoking \eqref{formH} gives $|H| = 2^n$ for $n$ even,
and $|H| = 2^{n-1}$ for $n$ odd.  This concludes the proof.
\end{proof}

We now combine the lemma with Thm.~\ref{thmfilledcubes} 
to obtain a lower bound for $|\Lambda^Q(\Sigma) \cap (2\Z_4)^n|$; 
the surface $\Sigma$ is embedded if and only if this lower bound is attained. 
For convenience we also include the assertion of Thm.~\ref{thmfilledcubes} 
in the following statement.  Recall that the complete surface~$\Sigma$ 
is periodic with lattice~$\Lambda(\Sigma)$
and generated from an initial surface~$\Sigma_0$ by Schwarz reflection. 
Moreover, $\Lambda^Q(\Sigma)$ and $S^Q$ denote quotient lattice 
and quotient group under the action of~$(4\Z)^n$.
\begin{theorem} \label{criterion}
  \ia\ For even dimension $n$ the following is equivalent:\\
  $\Sigma$ is embedded $\Longleftrightarrow |S^Q| 
  = 2^{n+2} \Longleftrightarrow |\Lambda^Q(\Sigma) \cap (2\Z_4)^n| = 4$, 
  whereas\\
  $\Sigma$ is not embedded 
  $\Longleftrightarrow |S^Q| \geq 2^{n+3} 
  \Longleftrightarrow |\Lambda^Q(\Sigma) \cap (2\Z_4)^n| \geq 8$.\\[1mm]
  \ii\ For odd dimension $n$ the following is equivalent:\\
  $\Sigma$ is embedded $\Longleftrightarrow |S^Q| = 2^{n+2} \Longleftrightarrow |\Lambda^Q(\Sigma) \cap (2\Z_4)^n| = 8$, whereas\\
  $\Sigma$ is not embedded $\Longleftrightarrow |S^Q| \geq 2^{n+3} \Longleftrightarrow |\Lambda^Q(\Sigma) \cap (2\Z_4)^n| \geq 16$.
\end{theorem}
\remarks
1. In particular, the lattice quotient contains non-trivial elements, that is, 
$\Lambda(\Sigma)\supsetneq (4\Z)^n$.\\[1mm]
2. Note that $S^Q$ has at least $m\ge 2n$ generators  
and so possibly $|S^Q|$ is as large as $2^m\ge 2^{2n}$, 
which grows much faster than $2^{n+2}$.
Therefore the likelihood of embedded surfaces should decrease
as $m$ or $n$ grow.\\[1mm]
3. There exist cases for which 
$\Lambda^Q(\Sigma)$ contains elements which are not in $(2\Z_4)^n$.
Therefore it is necessary to intersect $\Lambda^Q(\Sigma)$ with $(2\Z_4)^n$
for the above statements to hold.
In fact, this will be the case if and only if $\Sigma_0$ has nontrivial symmetries in $H$.
For instance, this holds if a single Schwarz reflection of $\Sigma_0$
is related to $\Sigma_0$ by a translation.

\section{Checking embeddedness of $\Sigma$ for specific curves $\Gamma$} 
\label{seccriterion}

\sloppypar
Thm.~\ref{criterion} completely characterizes embeddedness
in terms of the numbers $|S^Q|$ or $|\Lambda^Q(\Sigma)\cap (2\Z_4)^n|$.
It leaves open, however, how we can compute these numbers for
a given Jordan path~$\Gamma$.
The results of the present section will allow us to compute 
these numbers efficiently, 
so that the embeddedness problem can be decided.

\subsection{Test of embeddedness for even dimension}

Let us first note the equivalence
$$v \in \Lambda^Q(\Sigma) \cap (2\Z_4)^n \quad\Leftrightarrow\quad 
  v \in  (2\Z_4)^n \mbox{ and } C^v \mbox{ is filled} 
  \quad\Leftrightarrow\quad (v,0) \in S^Q\: ,
$$
where now $C^v$ represents a cube in the quotient.
This equivalence reduces our task to computing the number of elements 
of form $(v,0) \in S^Q$.

Consider an arbitrary element $(v,0)\in S^Q$ 
and represent it as the composition of $\ell\in\N$ generators of $S^Q$, 
\begin{equation}\label{vnull}
  (v,0)=s_1\circ\ldots\circ s_\ell, 
\end{equation}
where $s_i= (u_i,\rho^{\alpha(i)})$ for some $u_i\in\Z_4^n$
with $\rho^{\alpha(i)}$ as in \eqref{rhodef}, $1\le \alpha(i)\le n$.
Let $m_\beta\in 2\N$ be the number of indices~$i$ 
such that $\rho^{\alpha(i)}=\rho^{\beta}$. 

Restricted to the rotational part, 
composition in $S^Q$ is simply $\Z_2^n$-addition,
and so the rotational component $R(v,0)$ of \eqref{vnull} reads
\begin{equation}\label{mirhoi}
  0 = m_1 \rho^1 + \ldots + m_n\rho^n \in \Z_2^n.
\end{equation}
Due to the special form of $\rho^\beta$, see \eqref{rhodef},
addition of the $\beta$- and $\delta$-component of \eqref{mirhoi} 
gives $m_\beta=m_\delta\bmod 2$. 
So either all the $m_\beta$ are even or they are all odd.
Moreover, for $n$ even, any component of \eqref{mirhoi} proves
that the $m_\beta$ must in fact be even.
Let us concentrate first on the case that all $m_\beta$ are even,
and postpone the other case, arising only for odd dimension, till 
Sect.~\ref{ss:gendim}.

The representation $s_1\circ\ldots\circ s_\ell$ of $(v,0)$
contains an even number of generators in each coordinate direction.  
Let us permute the order of the $s_i$ so that they occur
in $\ell/2$ consecutive pairs such that each pair 
consists of generators pointing in the same direction
(same rotational part).
Since $S^Q$ is abelian (Lemma~\ref{commutative})
such a permutation leaves \eqref{vnull} unchanged.
Moreover, since the rotational part of each pair vanishes, 
the translational parts $v_1,\ldots,v_{\ell/2}$ of the pairs
have a composition $v=v_1+\ldots+v_{\ell/2}$ mod $\Z_4^n$.
This representation proves ``$\subseteq$'' of the following equality,
while the inclusion ``$\supseteq$'' is obvious.
\begin{proposition}\label{pr:evendimlattice} 
  If $n$ is even then 
  \[
    \Lambda^Q(\Sigma)\cap(2\Z_4)^n
    = \bigl\langle\{ T(s_i\circ s_{j}) \mid s_i,s_j\text{ generator 
              of $S^Q$ with } R(s_i)=R(s_j)\}\bigr\rangle \subset \Z_4^n.
  \]
\end{proposition}
\noindent
Consequently, for a given coordinate direction $\beta$
each pair of distinct edges in direction $\beta$ gives rise to a generator.
%
The generator can then be computed as follows (for arbitrary~$n$):
\begin{proposition} \label{lattice2} 
  Let $e_i$ and $e_j$ be two edges of $\Gamma$ in the $\beta$-direction, 
  and $s_i$, $s_j$ be the corresponding generators of~$S^Q$.
  Set 
\[
  \#(i,j;\delta) :=
  \text{ number of edges of $\Gamma$ in the $\delta$-direction 
    between $e_i$ and $e_{j}$.}
\]
  Then the vector $v:=T(s_i \circ s_j) \in (2\Z_4)^n$ has the components
$$ 
  v_\beta= 0 \quad\text{and,}\quad\text{ for $\delta\neq \beta$, }\quad
  v_\delta = 
  \begin{cases}
    0, & \text{if $\#(i,j;\delta)$ is even,} \\ 
    2, & \text{if $\#(i,j;\delta)$ is odd.} 
  \end{cases} 
$$ 
\end{proposition}
\noindent
Recall that we consider $\Gamma$ as cyclic, so that there are two 
choices of edges between $e_i$ and~$e_j$.  
But for each coordinate direction, $\Gamma$~contains an even number 
of edges, and so $\#(i,j;\delta) \modtwo$ becomes well-defined.
\begin{proof}
  We may assume $\beta=1$. Then the edges $e_i$ and $e_j$ are of form
  $$ 
    \textstyle
    e_i = \bigl\{(t, p_2, \ldots, p_n) \mid
     t\in \bigl[-\frac{1}{2},\frac{1}{2}\bigr] \bigr\},\quad
    e_j =  \bigl \{(t, q_2, \ldots, q_n) \mid
     t\in \bigl[-\frac{1}{2},\frac{1}{2}\bigr] \bigr\},
$$ 
where $p_2, \ldots, p_n,\, q_2, \ldots, q_n \in \bigl\{\pm \frac{1}{2} \bigr\}$.
By \eqref{formschwrefl}, the corresponding generators of $S^Q$ are 
\begin{align*}
 s_i &= \bigl( (0, 2 p_2 , \ldots, 2 p_n)\modfour, \,
    (0,1,\ldots,1) \bigr),\\
 s_j &= \bigl( (0,2 q_2,\ldots,2q_n)\modfour,\,(0,1,\ldots,1) \bigr) \: ,
\end{align*}
and we may write 
\begin{equation*}  
  s_i \circ s_j
  = \bigl( (0,2p_2-2q_2, \ldots, 2p_n-2q_n) \modfour, 0 \bigr)\, . 
\end{equation*}
So indeed $v_1=0$. 
Moreover, $p_\delta \neq q_\delta$ if and only if 
between $e_i$ and $e_j$ the Jordan path $\Gamma$ passes 
through an odd number of edges in the $\delta$-direction.  
In that case,  $2p_\delta-2q_\delta$ is $\pm 2$, while otherwise it vanishes.
\end{proof}

\subsection{Jordan paths $\Gamma$ with eight edges in $\R^4$}
\label{ss:r4m8}

We can now settle the embeddedness problem 
for the six Jordan paths of length~8 in~$\R^4$:
\begin{theorem} \label{eightedgesembedded}
  From the Jordan paths of Prop.~\ref{eightedges} exactly 
  $$
  \Gamma_3 = 1231\,4234, \qquad 
  \Gamma_4 = 1231\,4324, \quad \mbox{ and } \quad 
  \Gamma_6 = 1232\,1434 
  $$ 
  lead to an embedded surface upon Schwarz reflection.
\end{theorem}
\begin{figure}
\begin{center}
\includegraphics[scale=0.5]{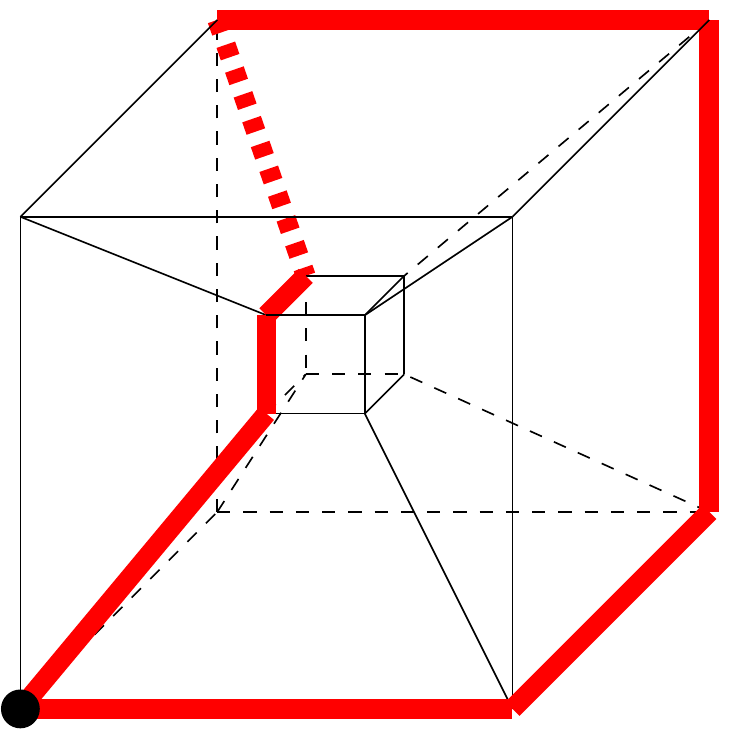} \hspace{3mm}
\includegraphics[scale=0.5]{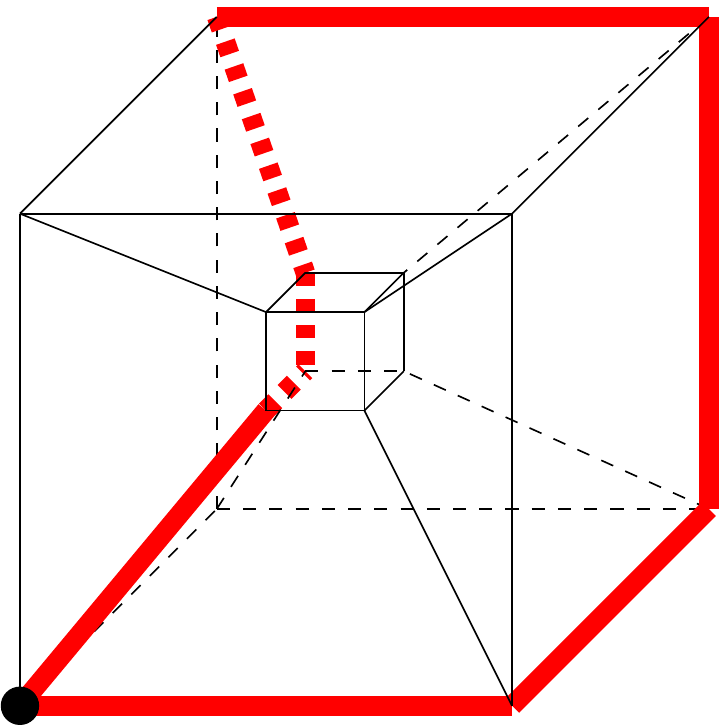} \hspace{3mm}
\includegraphics[scale=0.5]{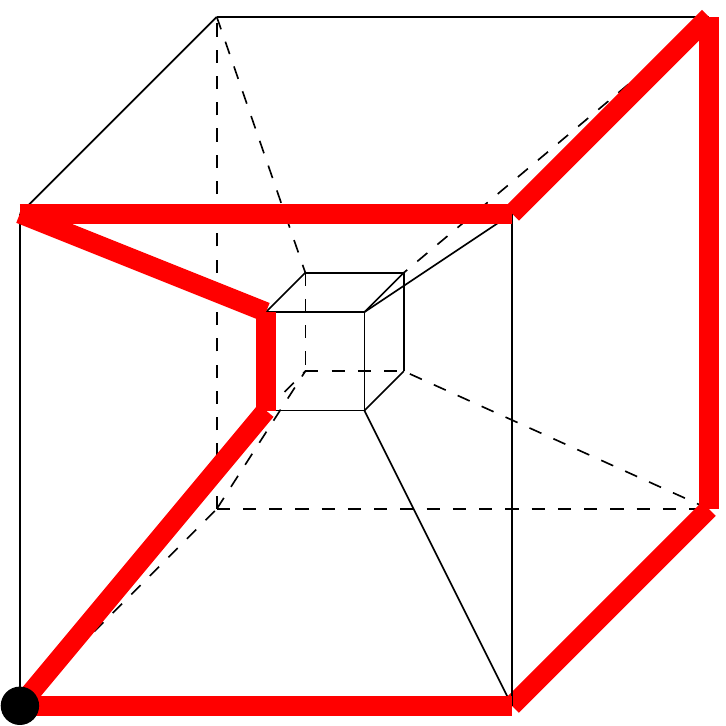} 
\end{center} 
\caption{\label{fi:R4embedded} 
  The Jordan paths with $8$ edges in $\R^4$ 
  which lead to embedded surfaces, namely
  $\Gamma_3=1231\,4234$, $\Gamma_4=1231\,4324$, and 
  $\Gamma_6=1232\,1434$.}
\end{figure}
\begin{proof} 
  For $\Gamma_1= 1231\,4243$ we claim 
  \begin{equation*}
    \Lambda^Q(\Sigma) \cap (2\Z_4)^4 
    = \left\langle \{(0,2,2,0), (2,0,2,2), (2,2,0,0), (0,2,0,0) \} \right\rangle.
  \end{equation*}
  To see this, note first that for each of the four coordinate 
  directions there is one pair of edges in~$\Gamma_1$.
  Therefore Prop.~\ref{pr:evendimlattice} gives four different nonzero generators,
  with components specified by Prop.~\ref{lattice2}.
  For $\beta=1$ this gives~$v_1=0$, while
  the other components follow from considering 
  the subsequences $23$ or~$4243$, enclosed by the pair of~$1$'s: 
  Each contains one edge in direction~$2$ and one in direction~$3$, 
  but an even number of $4$'s, so $v_2=v_3=2$ but~$v_4=0$.
  This verifies the first generator of the claim. 
  The same procedure, applied to $\beta=2,3,4$,
  gives the remaining three.

The four generators are independent, i.e.,
$\Lambda^Q(\Sigma) \cap (2\Z_4)^4 =
(2\Z_4)^4$,
and so $|\Lambda^Q(\Sigma) \cap (2\Z_4)^4| = 2^4 = 16$. 
Therefore Thm.~\ref{criterion}\ia\ implies that the surface generated from
$\Gamma_1$ has self-intersections.

We proceed similarly for the other Jordan curves.
For $\Gamma_2=1231\,4342$ we find
\begin{equation*}
  \Lambda^Q(\Sigma) \cap (2\Z_4)^4 
  = \left\langle \{(0,2,2,0), (2,0,0,0), (2,0,0,2), (0,0,2,0) \} \right\rangle,
\end{equation*}
and for $\Gamma_5=1234\,1234$ 
\begin{equation*}
  \Lambda^Q(\Sigma) \cap (2\Z_4)^4 
  = \left\langle \{(0,2,2,2), (2,0,2,2), (2,2,0,2), (2,2,2,0) \} \right\rangle.
\end{equation*}
Again, these groups agree with $(2\Z_4)^4$, and so 
Schwarz reflection generates surfaces with self-intersections.

On the other hand, we obtain for $\Gamma_3=1231\,4234$ 
\begin{align*}
  \Lambda^Q(\Sigma) \cap (2\Z_4)^4 
  &= \left\langle\{(0,2,2,0),(2,0,2,2),(2,2,0,2),(0,2,2,0) \}\right\rangle\\
  &= \left \langle \{(0,2,2,0), (2,0,2,2)\} \right \rangle,
\end{align*} 
for $\Gamma_4=1231\,4324$ 
\begin{align*}
  \Lambda^Q(\Sigma) \cap (2\Z_4)^4 
  &= \left \langle \{(0,2,2,0),(2,0,0,2),(2,0,0,2),(0,2,2,0)\}\right\rangle\\
  &= \left \langle \{(2,0,0,2), (0,2,2,0)\} \right \rangle,
\end{align*} 
and for $\Gamma_6=1232\,1434$ 
\begin{align*}
  \Lambda^Q(\Sigma) \cap (2\Z_4)^4 
  &= \left\langle\{(0,0,2,0),(0,0,2,0),(2,2,0,2),(0,0,2,0)\}\right\rangle\\
  &=  \left \langle \{(0,0,2,0),(2,2,0,2) \} \right \rangle.
\end{align*}  
  For these three cases
  we have $|\Lambda^Q(\Sigma) \cap (2\Z_4)^4| = 2^2 = 4$,
  and so by Thm.~\ref{criterion}\ia\ each of these three Jordan curves
  leads to an embedded surface.
\end{proof}

\subsection{Test of embeddedness for arbitrary dimension}\label{ss:gendim}

Let us now formulate Prop.~\ref{pr:evendimlattice}
for the case of arbitrary dimension.  
\begin{proposition} \label{lattice1}
  Let $s_1^0,\ldots,s_n^0$ be a fixed set of generators of $S^Q$,
  such that $R(s_\beta^0)= \rho^\beta$, and set 
  \begin{equation}\label{Lambda0}
    \Lambda^0:=\big\langle\left\{ T\bigl(s_{i}\circ {s}_{\beta}^0\bigr)\, | \, 
    s_i \mbox{ generator of } S^Q \text{ with } R(s_i)=\rho^\beta\! \right\} 
    \big\rangle
    \subset (2\Z_4)^n.
  \end{equation}
  Then, if $n$ is even, 
  $\Lambda^Q(\Sigma) \cap (2\Z_4)^n = \Lambda^0$,
  while for $n$ odd, 
\begin{equation}\label{latticeforoddn}
  \Lambda^Q(\Sigma) \cap (2\Z_4)^n 
  =  \big\langle \Lambda^0
    \cup  \{ T(s_1^0 \circ \ldots \circ s_n^0) \} \big\rangle .
\end{equation}
\end{proposition}
\noindent
For even dimension, the Proposition can
reduce the number of pairs to be considered
when compared with Prop.~\ref{pr:evendimlattice}:
If $k_\beta$ is the number of edges in direction~$\beta$
then the contributing generators for direction~$\beta$ are $k_\beta-1$ in \eqref{Lambda0},
while they are $\binom{k_\beta}2$ in Prop.~\ref{pr:evendimlattice}.
\begin{proof}
  We reason as we did for Prop.~\ref{pr:evendimlattice}.
  Consider first the case that all $m_\beta$ in \eqref{mirhoi} are even. 
  That is, the representation
  $s_1\circ\ldots\circ s_\ell$ contains an even number of generators
  in each coordinate direction~$\beta$.  
  As before let $\alpha(i)$ be such that $R(s_i)=\rho^{\alpha(i)}$.
  Since each $s_\beta^0$ is self-inverse and $S^Q$ is abelian
  we obtain
  \begin{equation}\label{sreprforevenm}
    s_1\circ\ldots\circ s_\ell
    =s_1\circ s^0_{\alpha(1)}\circ \ldots \circ s_\ell\circ s^0_{\alpha(\ell)};
  \end{equation}
  in particular, the translational part of \eqref{vnull} satisfies
  \[
    v=\sum_{i=1}^\ell T(s_i\circ s^0_{\alpha(i)}).  
  \] 
  The claim for even $n$ follows. 

\sloppypar
  Let us now consider the case that $n$ and all $m_\beta$ are odd.
  Then the very same argument shows
  \begin{equation}\label{sreprforoddm}
    s_1\circ\ldots\circ s_\ell
    =(s_1\circ s^0_{\alpha(1)})\circ\ldots\circ (s_\ell\circ s^0_{\alpha(\ell)})
     \,\circ\,(s_1^0\circ\ldots\circ s_n^0).
  \end{equation}
  Again let us take translational parts, noting that 
  $R(s_1^0\circ\ldots\circ s_n^0)=0$.
  Since any $(v,0)\in S^Q$ can be represented as in either 
  \eqref{sreprforevenm} or \eqref{sreprforoddm}
  we conclude ``$\subseteq$'' in \eqref{latticeforoddn},
  while ``$\supseteq$'' is obvious.
\end{proof}

For given $\Gamma$, Prop.~\ref{lattice2} allows us to compute
$|\Lambda^0|=2^k$ where $k$ is the number 
of independent generators of~$\Lambda^0$. 
For $n$ even, Prop.~\ref{lattice1} then gives
$|\Lambda^0| = |\Lambda^Q(\Sigma) \cap (2\Z_4)^n|$
and $\Sigma$ is embedded if and only if this number is $4$,  
by Thm.~\ref{criterion}\ia. 
On the other hand, for $n$ odd, 
either $|\Lambda^0| = |\Lambda^Q(\Sigma) \cap (2\Z_4)^n|$ or 
$2 |\Lambda^0| = |\Lambda^Q(\Sigma) \cap (2\Z_4)^n|$, 
depending on whether the exceptional element 
$T(s_1^0\circ \ldots \circ s^0_n)$ lies in $\Lambda^0$ or not. 
Invoking Thm.~\ref{criterion}\ii\ gives:
$|\Lambda^0| = 4$ implies $\Sigma$ is embedded, 
while $|\Lambda^0|\geq 16$ implies $\Sigma$ is not embedded. 
Thus only in the remaining case $|\Lambda^0|=8$ we need to calculate 
the exceptional element. 

\section{Classification of embedded periodic surfaces for $\R^4$}
\label{se:rfourcase}

Our goal is to determine all embedded $n$-periodic surfaces $\Sigma$ 
for $n=4$, with arbitrary number of edges $m$.  
We need to introduce two pieces of general theory first.

\subsection{Orientability of the extended surfaces}
\label{sectionorientable}

Assuming that $\Sigma_0$ is orientable, is its extension $\Sigma$ orientable?  
This will certainly be the case for an embedded surface $\Sigma$ 
with codimension~$1$, but it may fail for higher codimension.


Let us first observe that for any dimension $n$,
a Schwarz reflection changes the orientation of $\Sigma$.
So $\Sigma$ is non-orientable if and only if there exists an odd number
of generators of $S$ whose composition gives the identity.
This is the key to proving the following:
\begin{proposition} \label{proporient}
  Assume the initial surface $\Sigma_0$ is orientable.\\
  \ia\ Then for even dimension $n$ the surface $\Sigma$ 
  is orientable, and so is
  its quotient $\Sigma/ (\Lambda(\Sigma) \cap (2\Z)^n)= \Sigma^Q/\Lambda^0$.\\
  \ii\ For $n$ odd, $\Sigma$ is orientable if and only if 
  ${\Lambda^0} \neq \Lambda^Q(\Sigma) \cap (2\Z_4)^n$. 
  Furthermore, the quotient $\Sigma^Q/\Lambda^0$ 
  is orientable if and only if $\Sigma$ is orientable. 
  The surface $\Sigma/(\Lambda(\Sigma) \cap (2\Z)^n)$ is non-orientable.
\end{proposition}
\begin{proof}
\ia\ As was pointed out after \eqref{mirhoi},
for $n$ even the number of generators 
$\sum_{\beta=1}^n m_\beta$ of the identity element $(0,0)$ 
is even, and so $\Sigma$ is always orientable. 
Furthermore, an inspection of the proof of Thm.~\ref{4znlattice}
establishes that $\tau_{4 e_i}$ is the composition of four generators of $S$,
namely, $\tau_{4e_i}=s_1 \circ s_2 \circ s_1 \circ s_2$.
%
Furthermore, each element of $\Lambda^0$ is a composition
of an even number of generators of $S^Q$ by~\eqref{Lambda0}.
Therefore the quotient surface 
$\Sigma/ (\Lambda(\Sigma) \cap (2\Z)^n)$ is orientable, too. 

\noindent
\ii\ 
The following can be shown:
$\Sigma$ is not orientable if and only if there exists 
an odd number of generators of $S^Q$ whose composition is $(0,0)\in S^Q$. 
This holds for $n$ odd if and only if the exceptional generator 
$(s^0_1 \circ \ldots \circ s^0_n)$ of $\Lambda^Q(\Sigma) \cap (2\Z_4)^n$ 
does lie in ${\Lambda^0}$.
%
\end{proof}

We are not aware of embedded examples of non-orientable surfaces $\Sigma$
for odd dimension.  However, for $n=5$ the 
example $\Gamma = 145\,231\,425\,232$ generates a non-oriented
immersed surface~$\Sigma$.  It may well be that all embedded 
surfaces are orientable -- this is an open problem.

\subsection{Bound on the number of edges}

As pointed out before, the likelihood 
for surfaces~$\Sigma$ to be embedded decreases as 
the length $m$ of $\Gamma=\partial\Sigma_0$ increases. 
In fact, we can state an upper bound for the length of~$\Gamma$: 
\pagebreak[3]
\begin{proposition} \label{pr:bound}  
  Let the dimension $n$ be even. 
  If the number $m$ of edges of $\Gamma$ exceeds $4 (n-1)$, 
  then the surface $\Sigma$ cannot be embedded.
\end{proposition}
The bound is sharp, as the following
Jordan path of length $4(n-1)$ in~$\R^n$ shows:
\[
  \Gamma= 1\,\; 3\nearrow n \;\, 2\, \; n \searrow 3 \;\, 1\, \; 
     3 \nearrow n\, \; 2\, \; n \searrow 3
\]
%
Here, arrows denote an increasing or decreasing sequence of consecutive
positive integers.
By applying Thm.~\ref{criterion} and Proposition \ref{lattice1} 
it can be shown that $\Gamma$ leads to an embedded surface for all $n\ge 4$,
even or odd.

For the proof, and for later reference, let us show:
\begin{lemma} \label{generatorzero}
  Let $n$ be even and $\Sigma\subset\R^n$ be embedded.\\
  \ia\ Then $\Gamma$ cannot contain more than $4$ edges 
  in any given direction.\\
  \ii\ If\/ $\Gamma$ contains $4$ edges in the $\beta$-direction
  then each element $v\in\Lambda^Q(\Sigma) \cap (2\Z_4)^n$ 
  has component $v_\beta=0$.
\end{lemma}
\begin{proof}
  \sloppypar
  Assume there are $k$ distinct edges in some direction $\beta$,
  giving rise to generators $s_1,\ldots,s_k$ of $S^Q$.  
  Consider the compositions
  $s_{i1}:=s_{i} \circ s_{1}$ for $1\le i\le k$.
  Since the $k$ edges are parallel and distinct,
  the elements $s_{i1}$ have vanishing rotational part,
  but their translational parts are pairwise distinct.
  Also, $T(s_{i1})\in\Lambda^Q(\Sigma) \cap (2\Z_4)^n$
  by Prop.~\ref{lattice1}.
  For $\Sigma$ embedded, 
  $\Lambda^Q(\Sigma) \cap (2\Z_4)^n$ contains four elements
  by Thm.~\ref{criterion}\ia, and so $k\le 4$, thereby proving~\ia.

  Also, by Prop.~\ref{lattice2} the elements $s_{i1}$ 
  have vanishing $\beta$-co\-ordinate.
  If we assume $k=4$ and embeddedness then $s_{11},\ldots,s_{41}$
  make up for the four elements of $\Lambda^Q(\Sigma) \cap (2\Z_4)^n$.
  Therefore, any further element must agree with some
  $s_{11}, \ldots , s_{41}$, and so it must have vanishing $\beta$-component,
  as claimed in~\ii.
\end{proof}

\begin{proof}[Proof of the Proposition]
  Suppose $m>4(n-1)$ and $\Sigma$ were embedded. 
  That $m$ is even implies $m\ge 4n-2$.
  Since $6$ edges in the same direction are impossible by the lemma,
  $\Gamma$ must contain $4$ edges 
  in at least $n-1$ coordinate directions, say for
  $\beta=1, \ldots, {n-1}$.
  %
  By the lemma each element $v\in\Lambda^Q(\Sigma) \cap (2\Z_4)^n$ has 
	vanishing $\beta$-component for $\beta=1$ through $\beta=n-1$.  
  But there are only two such elements in $(2\Z_4)^n$,
  namely $(0,\ldots, 0,2)$ and $0$, a contradiction to Thm.~\ref{criterion}.
\end{proof}

A similar reasoning can be applied show that $m\le 8(n-3)+3\cdot 6$ 
holds for embedded $\Sigma$ when the dimension $n$ is odd.
Furthermore, we may replace $\Lambda^Q(\Sigma) \cap (2\Z_4)^n$ by $\Lambda^0$ 
in the reasoning above. 
This shows the bound  $m \le 4 (n-1)$ holds for all 
oriented embedded surfaces $\Sigma$ in odd dimensions $n$ as well.

\subsection{All embedded $4$-dimensional surfaces}

We can now classify all embedded $n$-periodic surfaces $\Sigma$ 
in dimension $n=4$ up to symmetry, see Figs.~\ref{fi:R4embedded}
and~\ref{abbkurven4D}.
\pagebreak[3]
\begin{theorem} \label{theoremR4}  
  There are exactly five different Jordan paths $\Gamma$ 
  generating embedded surfaces $\Sigma\subset\R^4$. 
  The surfaces are oriented, with oriented quotients.
  They are generated by the following Jordan paths:\\
  $\bullet$ For $m=8$ edges the Jordan paths
  \[
    \Gamma_3 = 1231\,4234, \quad \Gamma_4 = 1231\,4324, \quad  
    \Gamma_6 = 1232\,1434 
  \]
  \hphantom{$\bullet$}
  generate embedded surfaces of genus $9$.\\
  $\bullet$ For $m=10$ the path $\Gamma_7 := 123\,1413\,214$
  generates a surface of genus~$13$.\\
  $\bullet$ For $m=12$ the path $\Gamma_{8} := 123\,214\,123\,214$
  generates a surface of genus~$17$.\\
  Here, the genus is with respect to the lattice $\Lambda(\Sigma)\cap (2\Z)^4$.
\end{theorem}

\begin{proof}
  The case $m=8$ is covered by Thm.~\ref{eightedgesembedded},
  and $m>12$ contradicts Prop.~\ref{pr:bound}.
  Therefore we need to discuss the cases $m=10$ and $12$ only.

\begin{figure}[b]
\begin{center}
\includegraphics[scale=0.5]{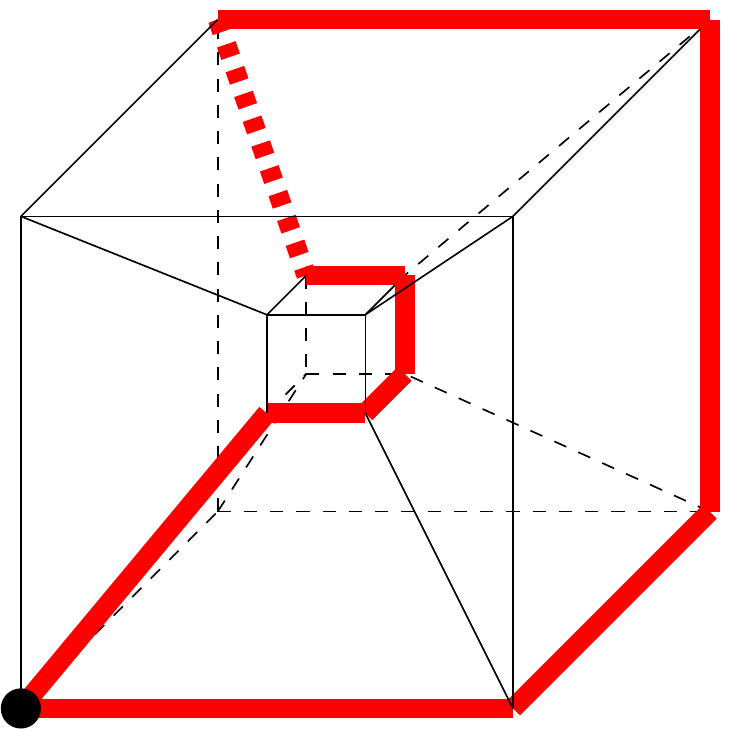}\hspace{5mm}
\includegraphics[scale=0.5]{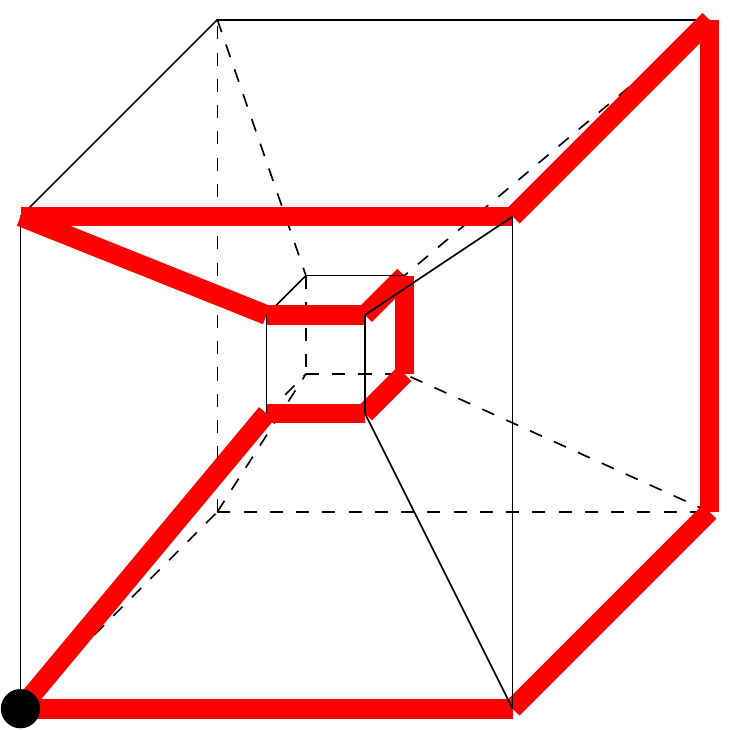}
\end{center}
\caption{The Jordan paths 
  $\Gamma_7=123\,1413\,214$ and $\Gamma_8=123\,214\,123\,214$,
  which lead to embedded surfaces in $\R^4$.\label{abbkurven4D}}
\end{figure}

\noindent Case $m=10$: 
  Then $\Gamma$ must have four edges in some direction, 
  which we assume to be~$1$,
  and two edges in the remaining three directions $\beta=2,3,4$. 
  Let $\Sigma$ be embedded.
  Then, by Lemma~\ref{generatorzero},
  every element in $\Lambda^Q(\Sigma) \cap (2\Z_4)^4$ 
  has the first coordinate zero.
  Moreover Prop.~\ref{lattice2} gives
  that each pair of edges in one of the $\beta$-directions, 
  $2\leq \beta \leq n$,
  encloses an even number of edges in the $1$-direction.
  Moreover, two edges in the same direction cannot be consecutive.

  Suppose that $\Gamma$ contains a pair of $1$'s which encloses
  a pair running in the same direction, let it be $2$. 
  Then $\Gamma$ must be of the form
  \[
    \Gamma= 1 \ldots 2  \cdots  2 \ldots 1 \cdots 1 \cdots 1 \cdots,
  \]
  where only the slots marked with low dots ($\ldots$) can be empty.
  Up to permutation, this lets $\Gamma$ be $12321\,41314$.
  We claim the lattice then contains the three independent directions
  $(0,0,2,0)$, 
  $(0,0,0,2)$, 
  $(0,2,0,2)$. 
  These are obtained by applying Prop.~\ref{lattice2}:
  The first vector arises for the pair of $2$'s,
  the second for the pair of $1$'s enclosing a $4$,
  the third for the pair of $3$'s.
  Thus $|\Lambda^Q(\Sigma) \cap (2\Z_4)^4|\ge 8$, meaning
  that $\Sigma$ has self-intersections by Thm.~\ref{criterion}\ia.
  
  In the other case, $\Gamma$ is of the following form, up to permutation,
  where again only the low dots can remain empty:
  \[
     \Gamma = 1 \ldots 2 \ldots 1 \cdots  1 \ldots 2 \ldots 1 \cdots
  \] 
  Reasoning as before this leaves us with 
  $\Gamma_* = 1231\, 4\, 1231\, 4$ or $\Gamma_7 = 1231\, 4\, 1321\, 4$,
  up to permutation.
  For $\Gamma_*$ again we get 3 independent generators (with a leading $0$), 
  implying $|\Lambda^Q(\Sigma) \cap (2\Z_4)^4|\ge 8$, 
  so that self-intersections arise.

  So we are left with $\Gamma_7$.
  For the direction $1$, we invoke Prop.~\ref{lattice1}, 
  select the first~$1$ for $s_1^0$, 
  and obtain (again by Prop.~\ref{lattice2}) 
  the three generators $(0,2,2,0)$, $(0,2,2,2)$, $(0,0,0,2)$.
  Moreover, for the directions $2,3,4$, we get a generator each,
  namely $(0,0,0,2)$, $(0,0,0,2)$, $(0,2,2,0)$.
  This gives 
  \[
    |\Lambda^Q(\Sigma) \cap (2\Z_4)^4| 
    = \bigl|\left \langle \{ (0,2,2,0), (0,0,0,2) \} \right \rangle\bigr|  
    = 4  .
  \]
  Thm.~\ref{criterion}\ia\ proves embeddedness.

\smallskip
\noindent Case $m=12$:
  For an embedded surface, Lemma~\ref{generatorzero}\ia\ gives that
  $\Gamma$ contains at most four edges in any given direction.
  Therefore, there are exactly two directions, say $1$ and $2$,
  for which $\Gamma$ contains four edges.
  By Lemma~\ref{generatorzero}\ii\ then all elements 
  from $\Lambda^Q(\Sigma) \cap (2\Z_4)^4$ 
  have their first and second coordinates equal to zero.
  Combining this fact with Prop.~\ref{lattice2} 
  gives that each pair of consecutive $1$'s
  encloses an even number of~$2$'s, and vice versa.
  Consequently $\Gamma$ can be 
  $$
    \Gamma_{*} = 1\,2\cdots 2 \cdots 2 \cdots 2\, 1 \cdots 1 \cdots 1 \cdots
    \quad\text{or}\quad
    \Gamma_{8} = 1\,2 \cdots 2\, 1 \cdots 1\, 2 \cdots 2 \, 1 \cdots,
  $$
  where all gaps must be non-empty.  However,
  $\Gamma_{*}$ is impossible as 4~remaining numbers cannot fill $6$~gaps.
  Let us show there is a unique way of filling the gaps of~$\Gamma_8$.
  By the argument used before
  there is an even number of $1$'s between a pair of $3$'s or $4$'s.
  Up to permutation,
  this identifies $\Gamma_8$ as claimed in the theorem.

  The last step of the proof is 
  to check $\Gamma_{8}$ leads to an embedded surface $\Sigma$. 
  By Prop.~\ref{lattice2} and Prop.~\ref{pr:evendimlattice} 
  the first and second coordinates of each element 
  from $\Lambda^Q(\Sigma) \cap (2\Z_4)^4$ is equal to zero.
  So Thm.~\ref{criterion}\ia\ proves embeddedness.
\smallskip

\sloppypar
  To compute the genus, let us first compute the 
  number of filled cubes in
  $\R^n/(\Lambda(\Sigma) \cap (2\Z)^4)$.
  The quotient $\R^n/(4\Z)^n$ contains $4^n$ cubes,
  from which, according to Thm.~\ref{criterion}\ia, 
  exactly $|S^Q|=2^{n+2}$ are filled.
  Now $n=4$ is even, and so again Thm.~\ref{criterion}\ia\ gives
  that $T^n$ has 
  $|\Lambda(\Sigma)^Q \cap (2\Z_4)^4|=4$ further translations.
  Therefore, the quotient $\R^n/(\Lambda(\Sigma) \cap (2\Z)^4)$ contains
  $2^{n+2}/4 = 2^n$ filled cubes, where $n=4$.

  We use the Euler formula to compute the genus.
  If $\Gamma$ has $m$ edges then 
  \[
    \chi\bigl(\Sigma/(\Lambda(\Sigma)\cap (2\Z)^4)\bigr) 
    = V-E+F 
    = 2^n \Bigl(\frac m4- \frac m2 + 1\Bigr)
    = 2^n \Bigl(1- \frac m4\Bigr) .
  \]
  For $n=4$ and $m=8$ this gives $\chi= -16 =2-2g$ or $g=9$,
  while for $m=10$ we have $\chi=
  -24$ and
  $g=13$, and for $m=12$ we have $\chi=-32$, that is, $g=17$.
  Orientability of these quotients follows from Prop.~\ref{proporient}.
\end{proof}

\remark
  It is natural to ask for the genus also with respect to 
  the full lattice of orientation preserving
  translations of $\Sigma$. 
  In the following we show that the genus with respect to this
  lattice agrees with the genus stated in the theorem. 
  The only exception is the surface generated by $\Gamma_8$,
  whose genus reduces to~$9$, provided its initial surface $\Sigma_0$ 
  has all the symmetries of~$\Gamma_8$.
  
  Let us reason for this claim.
  Suppose $\Sigma$ has an orientation preserving translation $\tau$. 
  If $\tau$ is not in the lattice $\Lambda(\Sigma) \cap (2\Z)^4$ 
  then there is an element $(v,\rho)\in S$ such that
  $v=\tau$ and $\rho\not=0$.  
  Hence $(v-\tau,\rho) = (0, \rho)$ is a non-trivial symmetry of $\Sigma_0$,
  and in particular a non-trivial symmetry of $\Gamma$.
  Note that by construction, $\rho$ preserves coordinate directions.

  We inspect the five Jordan paths for such symmetries.
  The paths $\Gamma_3$ and $\Gamma_4$ do not admit any such symmetries.
  The path $\Gamma_6$ is only symmetric under $\rho=(0,0,1,0)$, a reflection
  in the hyperplane $\{x_3=0\}$.  
  Using $\rho^\beta$ as in~\eqref{rhodef} we find the representation
  $\rho=\rho^1\circ\rho^2\circ\rho^4$,
  that is, we represent $(v,\rho)$ with an odd number
  of generators of~$S$. 
  Therefore $\Gamma_6$ has no orientation preserving symmetry.
  Similarly, $\Gamma_7$ only has the symmetry $\rho=(0,0,0,1)$
  which again is orientation reversing.
  However, the path $\Gamma_8$ has two orientation reversing symmetries, 
  $(0,0,1,0)$, $(0,0,0,1)$, 
  whose composition $(0,0,1,1)$ is orientation preserving.

\section{Embedded surface families for all dimensions} \label{ss:families}

As a consequence of Thm.~\ref{criterion}, if a surface is embedded
then the density of filled cubes must be $4\cdot 2^n/4^n = (1/2)^{n-2}$.
As this number decreases in $n$ the question arises
whether embedded surfaces can exist for all dimensions $n$.
We will answer it in the affirmative.

To give an example let us initiate from the Jordan path 
$123\,123$ in $\R^3$.  In the minimal setting
it generates 
the Schwarz-$D$-surface, with
$$ 
  |{\Lambda^0}(\Sigma)| = 4 \quad \mbox{ and} \quad 
  |\Lambda^Q(\Sigma) \cap (2\Z_4)^n| = 8 \, ,
$$
where $\Lambda^0(\Sigma)$ generates the face centred cubic lattice.

To generalize this example to arbitrary dimension $n\ge 3$
we consider a surface $\Sigma_n\subset\R^n$ generated by 
\[
  \Gamma_n= 1\,2\,3\,\nearrow n\; 1\,2\; n \searrow 3,
\] 
where again arrows denote an increasing or decreasing sequence 
of consecutive integers.
We claim that $|\Lambda^0(\Sigma_n)|=4$.  
To verify this, let us calculate the generators in the $\beta$-direction 
using Prop.~\ref{lattice2}:
The generators for $\beta=1$ and $2$ are
$(0,2,2,2,\ldots,2)$ and $(2,0,2,2,\ldots,2)$,
while all $\beta\ge 3$ give the same generator $(2,2,0,0,\ldots,0)$.
The last generator is the sum of the first two generators,
and so our claim is proven.
Combining Prop.~\ref{lattice1} with 
Thm.~\ref{criterion} proves that $\Sigma_n$ is embedded.
In particular, for $n$ odd, we must have
$|\Lambda^Q(\Sigma) \cap (2\Z_4)^n| = 8$, and so
$\Sigma_n$ is orientable in any dimension, by Prop.~\ref{proporient}.

We can apply the same procedure systematically to construct
embedded surfaces in higher dimensions:
\begin{proposition}\label{pr:series}
  Let $\Gamma_n= (\gamma(1)\cdots\gamma(m))$ be a Jordan path in dimension 
  $n$ with $|{\Lambda^0}(\Sigma)|= 4$.
  For $N>n$ and any $1\le \beta\le n$ let the Jordan path $\Gamma_{N,\beta}$ 
  in $\R^N$ be obtained from~$\Gamma$ by 
  replacing every $\gamma(i)=\beta$ in an alternating way 
  by the sequences $\beta\, (n+1) \cdots N$ or $N \cdots (n+1) \,\beta$. 
  Then ${\Gamma_{N,\beta}}$ generates embedded, 
  orientable surfaces $\Sigma_{N,\beta}$ in $\R^N$.
\end{proposition}
The proof amounts to checking that application of Prop.~\ref{lattice1} 
leads to a 1-1 correspondence of the generators of $\Lambda^0$
for $\Sigma$ and $\Sigma_{N,\beta}$.

\examples
  The Jordan paths $1212$ in $\R^2$ and the paths 
  $123\,123$ and $1232\,1232$ in $\R^3$ 
  give the following Jordan curves generating
  embedded surfaces $\Sigma$ for arbitrary dimension $n$:
  For $1\leq \beta < n$
  \[ 
    \Gamma_a:= 1\nearrow n\; \beta \searrow 1\; n\searrow (\beta+1)
  \]  
  and for $1\leq \alpha<\beta < n$ 
  \begin{align*}
    &\Gamma_b:= 1\nearrow  n\; \alpha \searrow 1 \; \beta \searrow (\alpha+1)\; n \searrow (\beta+1)\\
    &\Gamma_c:= 1\nearrow n \; \beta \searrow (\alpha+1) \; \alpha \searrow 1\;  (\alpha+1) \nearrow \beta  \; n \searrow (\beta+1)\; \beta \searrow (\alpha+1)
    %
\end{align*}
lead to families of embedded, orientable surfaces $\Sigma$ in~$\R^n$.
We have obtained a set of Jordan paths which grows quadratically in~$n$.

We would like to point out that all embedded surfaces $\Sigma$ in~$\R^3$ 
and $\R^4$, given by Prop.~\ref{dim3} and Thm.~\ref{theoremR4}, 
belong to one of these families. 
It remains open whether there are more examples of 
embedded surfaces $\Sigma$ for dimensions~$n>4$. 

\section{Minimal surfaces} \label{se:min}

In the present section we quote standard results from minimal surface theory.
First we verify that Plateau solutions generate initial surfaces.

\begin{proposition}\label{pr:plateausol}  
  \ia\ \emph{Plateau solution:}
  For each Jordan path $\Gamma$ in $\R^n$ 
  there exists a continuous map 
  $f\colon \overline D\to C=[-\frac12,\frac12]^n$,
  where $D$ is the open unit disk in~$\R^2$,
  such that $f$ is a $C^\infty$ conformal harmonic immersion from~$D$
  to $(-\frac12,\frac12)^n$.  Moreover, 
  the restriction $f\colon \partial D\to \Gamma$ is bijective.\\
  \ii\ \emph{(Interior) branch points:}
  $f$ is a regular surface
  if there exists a coordinate direction $1\le \alpha\le n$ such that 
  $\Gamma$ contains only two edges in the $\alpha$-direction.\\
  \iii\ \emph{Embedding:} 
  If there exist
  $1\le \alpha\not=\beta \le n$ such that the projection
  $\pi_{\alpha\beta}\colon \Gamma\to \partial([0,1]^2)$ of $\Gamma$ to the 
  $\alpha\beta$-plane is monotone
  then $f$ is an embedding.
\end{proposition}
\begin{proof}
  \ia\ The existence of the minimal disk is standard, 
  see e.g.~\cite[Thm.~7.1]{oss}.
  Since $\Gamma$ is not contained in a face of the cube~$C$
  the maximum principle proves that $D$ is mapped into its interior.
  
  \noindent
  \ii\ 
  Consider the harmonic function $h:=\langle f-q,\,v\rangle$,
  where $q,v\in\R^n$.
  Since $\Gamma$ is not contained in a hyperplane the set $h^{-1}(0)$ does
  not coincide with the entire disk~$D$, 
  and so the analyticity of~$h$ implies
  $h^{-1}(0)$ is the union of proper smooth curves, 
  which meet at the isolated branch points of~$h$.

  Suppose $p\in D$ is a branch point of $f$, with image $q:=f(p)$. 
  Then $h$ also has a branch point at~$p$
  and there are at least four arcs in $h^{-1}(0)$ emanating from~$p$.
  Moreover, a pair of consecutive arcs must have distinct endpoints 
  on~$\partial D$, by the maximum principle.  
  Thus there are at least four distinct points on $\partial D$
  which are zeros of~$h$.  
  The injectivity of~$f$ restricted to $\partial D$ implies 
  these correspond to four distinct points of $\Gamma$ 
  on a hyperplane through~$q$.
  However, according to our assumption, for the choice $v:=e_\alpha$ 
  there are at most two such points, a contradiction.

  \noindent
  \iii\ This is a consequence of Rado's lemma: 
  For the special case that a Jordan curve has a 1-1~projection 
  onto the boundary of a convex set, it is proven in \cite[Theorem 7.2]{oss} 
  that the minimal surface can be represented as a graph 
  over the $\alpha\beta$-plane, in particular it is embedded.
  For our more general case with possible vertical segments, 
  the arguments of Nitsche \cite[\paragraph 401]{nit} (given there for $n=3$) 
  prove the claim.
\end{proof}
Second, we show that for the contours which generate embedded surfaces,
these specifically generate embedded minimal surfaces.
\begin{theorem}
  Suppose $\Gamma$ is 
  one of the five Jordan curves in $\R^4$ of Thm.~\ref{theoremR4} 
  or, more generally,
  contained in one of the families $\Gamma_a,\Gamma_b,\Gamma_c\subset\R^n$
  defined in Sect.~\ref{ss:families}.\\
  \ia\ Then $\Gamma$ bounds an initial minimal surface~$\Sigma_0$,
  in particular $\Sigma_0$ is embedded.\\
  \ii\ The completion by Schwarz reflection $\Sigma$ of~$\Sigma_0$ 
  is an embedded minimal surface.
\end{theorem}
\begin{proof}
  \ia\ Let us demonstrate this on the example of $\Gamma_3$
  with $\gamma =1231\,4234$: 
  For $\alpha\beta=12$ only edges
  in direction $1$ and $2$ are left after projection. 
  Therefore, $\pi_{12}(\Gamma_3)$ is the cycle $1212$, 
  meaning that the projection is a monotone parameterization 
  of the square boundary.  This verifies the condition for an embedding  
  required in Prop.~\ref{pr:plateausol}.
  Inspection of the other Jordan paths is straightforward.

  \noindent
  \ii\ Schwarz reflection continues a conformal harmonic immersion 
  across the interior of a straight edge as such, see \cite[Lemma 7.3]{oss}.  
  At the vertices, the result of the reflections is an isolated
  singularity of the harmonic map, hence removable.

  By \ia, the surface $\Sigma_0$ is regular in the interior. 
  But $\Sigma_0$ is also free of branch points at the boundary:
  The initial surface~$\Sigma_0$ is contained 
  in the cube~$C$ and so has density~$1/2$ at interior points of the edges, 
  or~$1/4$ at vertices, 
  so the density of the reflected surface is~$1$ everywhere.
  %
  That is, the extended surface $\Sigma$ is free of branch points. 
  Then the statements of Thm.~\ref{theoremR4} 
  or Prop.~\ref{pr:series} guarantee the embeddedness of~$\Sigma$.
\end{proof}

Our statements leave the uniqueness problem open.
Note that for higher codimension there are known examples
of Jordan curves bounding two different graphs~\cite{laos}.
As a consequence, it is not evident that symmetries of~$\Gamma$
are inherited by a minimal disk~$\Sigma_0$.

\end{document}